\documentclass[10pt,a4paper]{amsart}
\usepackage{times}
\usepackage{amsmath}
\usepackage{graphicx}
\usepackage{subfigure}
\usepackage{xcolor}
\numberwithin{equation}{section}
\newcommand{\mathp}{\mathcal{P}}
\newcommand{\mo}{\mathcal{O}}
\newcommand{\e}{\epsilon}
\newcommand{\R}{\text{Re }}
\newcommand{\I}{\text{Im }}

\newtheorem{theorem}{Theorem}
\newtheorem*{remark}{Remark}
\newtheorem{lemma}{Lemma}

\begin{document}

\title[Asymptotics of the SVD of the truncated Hilbert transform]{Asymptotic analysis of the SVD for the truncated Hilbert transform with overlap}
\author[R Al-Aifari, M Defrise, A Katsevich]{Reema Al-Aifari$^1$ \and Michel Defrise $^2$ \and Alexander Katsevich$^3$}
\thanks{$^1$ Department of Mathematics, Vrije Universiteit Brussel, Brussels B-1050, Belgium}
\thanks{$^2$ Department of Nuclear Medicine, Vrije Universiteit Brussel, Brussels B-1050, Belgium}
\thanks{$^3$ Department of Mathematics, University of Central Florida, FL 32816, USA}

\begin{abstract}
The truncated Hilbert transform with overlap $H_T$ is an operator that arises in tomographic reconstruction from limited data, more precisely in the method of Differentiated Back-Projection (DBP). Recent work \cite{a-k} has shown that the singular values of this operator accumulate at both zero and one. 
To better understand the properties of the operator and, in particular, the ill-posedness of the inverse problem associated with it, it is of interest to know the rates at which
the singular values approach zero and one. In this paper, we exploit the property that $H_T$ commutes with a second-order differential operator $L_S$ and the global asymptotic behavior of its eigenfunctions to find the asymptotics of the singular values and singular functions of $H_T$.

\end{abstract}
\maketitle

\section{Introduction}\label{sec:intro}
In 2D or 3D Computerized Tomography (CT), an image of an object is reconstructed from measurements that can be modeled as Radon transform or cone beam transform data, respectively. Typically, a source emitting a beam of X-rays rotates around the object and a detector measures the attenuation of the X-ray beam after it traverses the object. When measurements from a sufficiently dense set of rays crossing the object are collected, standard techniques (e.g. Filtered Back-Projection) allow for stable reconstruction \cite{nat}.

In the case of limited data, e.g. when only measurements from an angular range less than $180$ degrees are available or when only a strict subset of the object support is illuminated from all directions, reconstruction becomes more difficult. While these cases can occur in practice (for example, with an oversize patient), reconstruction from limited data may also allow to reduce the radiation dose to which patients are exposed. 

The Differentiated Back-Projection (or DBP), a method based on a result by Gelfand and Graev \cite{gg}, allows to identify a class of limited data configurations, such that reconstruction is still possible. It is based on the reduction of the 2D or 3D problem to a family of 1D problems. These consist of the reconstruction of a compactly supported function in 1D from its partially known Hilbert transform. The application of the Gelfand-Graev formula to tomography was first introduced by Finch \cite{finch} and later made explicit for 2D \cite{noo, yyww, zps} and for 3D \cite{pnc, yzyw, zhuang, zp}.

In 2D, the Differentiated Back-Projection reduces the reconstruction problem to a family of 1D problems that can be formulated as inverting operators of the form $\mathp_{\Omega_1} H \mathp_{\Omega_2}$, where $H$ is the Hilbert transform on $L^2(\mathbb{R})$, $\Omega_1$, $\Omega_2$ are finite intervals on $\mathbb{R}$ and $\mathp_{\Omega}$ is the projection operator $(\mathp_{\Omega} f) (x) = f(x)$ if $x \in \Omega$, $(\mathp_{\Omega} f)(x) = 0$ otherwise. If $\Omega_2 \subset \Omega_1$, i.e. the Hilbert transform is measured on an interval covering the support of the object, the inversion of $\mathp_{\Omega_1} H \mathp_{\Omega_2}$ is well-posed and an explicit inversion formula is known \cite{tricomi}.

In general, when $\Omega_2 \not\subset \Omega_1$ the inversion of $\mathp_{\Omega_1} H \mathp_{\Omega_2}$ has turned out to be \textit{severely} ill-posed. Thus, it is of interest to study the singular value decomposition (SVD) of such operators. The SVD in the case of $\Omega_1 \subset \Omega_2$, which occurs in the so-called \textit{interior problem}, has been studied in \cite{kat2}. The SVD of the truncated Hilbert transform \textit{with a gap}, which describes the case $\Omega_1 \cap \Omega_2 = \emptyset$ has been the subject of \cite{kat1}. For both cases, the asymptotic behavior of the singular values and singular functions has been found in \cite{k-t}.  

This paper concerns a different setup, the truncated Hilbert transform \textit{with overlap} $H_T$. This is the case when the two intervals overlap, i.e. $\Omega_1=[a_1,a_3]$, $\Omega_2=[a_2,a_4]$ for real numbers $a_1 < a_2 < a_3 < a_4$. For this case, a uniqueness and \textit{pointwise} stability result for the inversion was obtained in \cite{defrise}. The SVD of the truncated Hilbert transform \textit{with overlap} has been characterized in \cite{a-k}, where it is shown that the singular values of $H_T$ accumulate at both $0$ and $1$, where the accumulation point $0$ causes the ill-posedness of inverting the operator $H_T$. Motivated by this result, this paper studies the asymptotic behavior of the singular values and singular functions of $H_T$. One of the main results we present here is an explicit expression for the exponential decay of the singular values of $H_T$ to zero, yielding the \textit{severe} ill-posedness of the underlying problem. 

The paper is organized as follows: Section \ref{prelim} starts with an overview of the results obtained in \cite{a-k}, that will be used in the sequel. In Section \ref{twoacc} we show an intermediate result on the eigenvalues of a differential operator that is related to the operator $H_T$ in a sense to be defined in Section \ref{prelim}. Next, Section \ref{sec:procedure} gives an outline and description of the approach used to find the asymptotic behavior of the SVD. In Section \ref{sec:asymp-0}, the asymptotic behavior of the SVD is derived for the subsequence of singular values accumulating at zero. We use this result together with a symmetry property in Section \ref{sigma-1} to obtain the asymptotics for the case where the singular values tend to $1$. We conclude by comparing the theoretical results obtained from the asymptotic analysis with a numerical example in Section \ref{sec:numerics}.

\section{Preliminaries}\label{prelim}
In \cite{a-k} we have analyzed the spectrum of the operator $H_T^* H_T$, where $H_T: L^2([a_2,a_4]) \to L^2([a_1,a_3])$ is the truncated Hilbert transform \textit{with overlap} defined for any fixed four real numbers $a_1 < a_2 < a_3 < a_4$ to be the following operator
\begin{equation}\label{tht}
(H_T f)(x) := \frac{1}{\pi} \text{p.v. } \int_{a_2}^{a_4} \frac{f(y)}{y-x} dy, \quad x \in (a_1,a_3)
\end{equation}
where $\text{p.v.}$ stands for the principal value. 

By relating $H_T$ to a self-adjoint extension of a differential operator with which it commutes, we found that the singular values of $H_T$ accumulate (only) at $0$ and $1$, where $0$ and $1$ themselves are not singular values. A natural question that then arises is the asymptotic behavior of the singular values, i.e. the convergence rates of the accumulation at $0$ and $1$. Especially in view of the ill-posedness of the inversion of $H_T$, it is important to ask how fast the singular values decay to zero. 

To answer this question, we will need to consider the singular value decomposition $\{ f_n, g_n; \sigma_n\}$, $n \in \mathbb{Z}$,  of $H_T$: 
\begin{align}
H_T f_n &= \sigma_n g_n, \label{svd1} \\
H_T^* g_n &= \sigma_n f_n, \label{svd2}
\end{align}
and study the asymptotic behavior of the singular functions $f_n$ and $g_n$ to find the asymptotics of $\sigma_n$. For the indices of the singular values we choose the convention $n \to + \infty$ for $\sigma_n \to 0$ and $n \to - \infty$ for $\sigma_n \to 1$.

In what follows, we briefly summarize results found in \cite{a-k}, to which we refer for details and proofs. By the commutation property, $\{ f_n \}_{n \in \mathbb{Z}}$ are the eigenfunctions of the differential operator $L_S$ that we define by first introducing
\begin{equation}
L(x,d_x) \psi(x):= (P(x)\psi'(x))'+2 (x-\sigma)^2 \psi(x)
\end{equation}
where
\begin{equation}
P(x) = \prod_{j=1}^4 (x-a_j), \quad  \sigma = \frac{1}{4}\sum_{j=1}^4 a_j.
\end{equation}
Let $D_{\max}$ denote the maximal domain on $(a_2,a_3) \cup (a_3,a_4)$ associated with $L(x,d_x)$ given by
\begin{align}
D_{\max} := \{ &\psi: (a_2,a_3) \cup (a_3,a_4) \rightarrow \mathbb{C}: \psi_{2,3}, P \psi_{2,3}' \in AC_{loc}(a_2,a_3), \label{dmax}\\
& \psi_{3,4}, P \psi_{3,4}' \in AC_{loc}(a_3,a_4); \psi, L \psi \in L^2([a_2,a_4]) \}, \nonumber
\end{align} 
where $\psi_{2,3}$, $\psi_{3,4}$ denote the restrictions of $\psi$ to $(a_2,a_3)$ and $(a_3,a_4)$, respectively, and $AC_{loc}(I)$ stands for the space of locally absolutely continuous functions on $I$.
Furthermore, we introduce the notation $a_j^{\pm} = \lim\limits_{\e \to 0^{\pm}} a_j+\e$ and the Lagrange sesquilinear form of two functions $u$, $v$:
$$ [u,v] := uP\overline{v}'-\overline{v}Pu'.$$
Then, the realization $L_S: D(L_S) \rightarrow L^2([a_2,a_4])$ of $L(x,d_x)$ on the domain 
\begin{align}
D(L_S) := \{ \psi \in D_{\max}: [\psi,u](a_2^+) &= [\psi,u](a_4^-) = 0, \label{ds}  \\
[\psi,u](a_3^-) &= [\psi,u](a_3^+), [\psi,v](a_3^-) = [\psi,v](a_3^+) \} \nonumber
\end{align}
with the following choice of maximal domain functions $u, v \in D_{\max}$
\begin{align}
u(y) &:= 1, \label{u} \\
v(y) &:= \sum_{i=1}^4 \prod_{\substack{j \neq i \\ j \in \{1, \dots, 4 \}}} \frac{1}{a_i-a_j} \ln|y-a_i|, \label{v}
\end{align} 
is self-adjoint. The spectrum of $L_S$ is real and discrete and the left singular functions $f_n$, $n \in \mathbb{Z}$, of $H_T$ are the eigenfunctions of $L_S$:
\begin{equation}\label{eigen-ls}
L_S f_n = \lambda_n f_n
\end{equation}
and form an orthonormal basis of $L^2([a_2,a_4])$. For the differential operator $\tilde{L}_S: D(\tilde{L}_S) \subset L^2([a_1,a_3]) \rightarrow  L^2([a_1,a_3])$, defined in the same way as $L_S$, but with $a_2, a_3, a_4$ replaced by $a_1,a_2,a_3$ in the definitions \eqref{dmax} and \eqref{ds}, we also obtain

\begin{equation}\label{eigen-tilde-ls}
\tilde{L}_S g_n = \lambda_n g_n.
\end{equation}
Here, $g_n$ are the right singular functions of $H_T$ from above. The eigenvalues $\lambda_n$ in \eqref{eigen-ls} and \eqref{eigen-tilde-ls} coincide.

From the theory of Fuchs-Frobenius, it follows that the points $a_i$ are \textit{regular singular} and that the two linearly independent solutions to $(L-\lambda) \psi = 0$ in a neighborhood of $a_i^+$ or $a_i^-$ are given by
\begin{align}
\psi_1(x) &= \sum_{n=0}^{\infty} b_n (x-a_i)^n \label{dof1} \\
\psi_2(x) &=  \sum_{n=0}^{\infty} d_n (x-a_i)^n + \ln |x-a_i| \psi_1(x) \label{dof2}
\end{align}
where the coefficients $d_n$ are different to the left and to the right of $a_i$.
This allows to simplify the characterization of the eigenfunctions $f_n$, $n \in \mathbb{Z}$:

A function $f \in L^2([a_2,a_4])$ is an eigenfunction of $L_S$ if and only if, 
\begin{itemize}
\item[--] it solves $L f = \lambda f$ for some $\lambda \in \mathbb{C}$, \\
\item[--] it is bounded at $a_2^+$ and at $a_4^-$, \\
\item[--] it is of the form $\phi_{11}(x) + \ln |x-a_3| \cdot \phi_{12}(x)$ at $a_3^-$ and \\
\item[--] of the form $\phi_{21}(x) + \ln |x-a_3| \cdot \phi_{22}(x)$ at $a_3^+$ and\\
\item[--] with analytic functions $\phi_{ij}$ such that $\phi_{11}(x)$ matches $\phi_{21}(x)$ continuously at $a_3$ and $\phi_{12}(x)$ matches $\phi_{22}(x)$ continuously at $a_3$, i.e.
\begin{align}
\lim_{x \to a_3^-} \phi_{11}(x) &= \lim_{x \to a_3^+} \phi_{21}(x) \label{tc1} \\
\lim_{x \to a_3^-} \phi_{12}(x) &= \lim_{x \to a_3^+} \phi_{22}(x) \label{tc2}
\end{align}
\end{itemize}
We refer to \eqref{tc1}, \eqref{tc2} as transmission conditions at the point $a_3$. 

At $a_i$ an eigenfunction $g$ of $\tilde{L}_S$ satisfies the same conditions that an eigenfunction of $L_S$ has at $a_{i+1}$, $i=1,2,3$.

\section{The spectrum of $L_S$ has two accumulation points}\label{twoacc}

In \cite{a-k}, we have shown that the operator $(L_S-i)^{-1}$ is compact. Hence, the spectrum of $L_S$ is discrete and the only possible accumulation points are $\lambda_n \to \pm \infty$, $n \in \mathbb{Z}$. As we will see in the following sections, deriving the asymptotics of the singular values $\sigma_n$ of $H_T$ for just one of the two possible accumulation points for $\lambda_n$ results in only one accumulation point of $\sigma_n$. More precisely, $\lambda_n \to + \infty$ leads to $\sigma_n \to 0$ and $\lambda_n \to -\infty$ to $\sigma_n \to 1$. Since we have shown in \cite{a-k} that both $0$ and $1$ are accumulation points of the spectrum of $H_T^* H_T$, this suggests that the eigenvalues $\lambda_n$ of $L_S$ accumulate at both $+\infty$ and $- \infty$. 

For self-adjoint realizations of $L(x,d_x)$ on an interval where the function $P(x)$ is negative, the spectrum of this self-adjoint realization is bounded below. Since in the case of $L_S$, we consider $P(x)$ on $(a_2,a_4)$, i.e. on an interval on which $P$ changes sign, it seems intuitive to assume that the spectrum of $L_S$ is unbounded from below and from above. 

Indeed, in the case where $P(x)$ changes sign and $1/P(x)$ is locally integrable on $(a_2,a_4)$, standard results in Sturm-Liouville theory state that the spectrum of the resulting differential operator is unbounded from below and from above, \cite{kong}. However, local integrability of $1/P(x)$ is not the case for $L_S$. In order to show the unboundedness from below and from above of the spectrum of $L_S$, we construct two sequences of functions $u_n \in D(L_S)$, $n \in \mathbb{N}$, supported on $[a_2,a_3]$ and $v_n \in D(L_S)$, $n \in \mathbb{N}$, supported on $[a_3,a_4]$ for which 
\begin{align}
\langle L_S u_n, u_n \rangle / \langle u_n, u_n \rangle & \to -\infty, \label{un} \\
\langle L_S v_n, v_n \rangle / \langle v_n, v_n \rangle & \to +\infty, \label{vn}
\end{align}
as $n \to \infty$. For $I \subset \mathbb{R}$, let $\chi_I$ denote the characteristic function on $I$ and define $w_1(x) = \chi_{[a_2,a_3]}(x)(x-a_2) (a_3-x)$ and $w_2(x) = \chi_{[a_3,a_4]}(x-a_3) (a_4-x)$. Then, we choose the functions $u_n$ and $v_n$ to be
\begin{align*}
u_n(x) &:= w_1(x) \cos(nx), \\
v_n(x) &:= w_2(x) \cos(nx).
\end{align*}
From $(P(x) u_n'(x))' = - P(x) w_1(x) n^2 \cos(nx) + \mo(n)$, we obtain
\begin{align}
\langle L_S u_n, u_n \rangle &= -  n^2 \int_{a_2}^{a_3} P(x) w_1^2(x) \cos^2(n x) dx + \mo(n), \label{Lu} \\
 & \leq - n^2 (a_2-a_1) (a_4-a_3) \int_{a_2}^{a_3} w_1^3(x) \cos^2(n x) dx + \mo(n) \nonumber
\end{align}
A direct computation yields
\begin{equation*}
\int_{a_2}^{a_3} ((x-a_2)(a_3-x))^3 \cos^2(nx) dx = \frac{(a_3-a_2)^7}{280} + \mo(n^{-4})
\end{equation*}
so that the integral on the right-hand side in \eqref{Lu} is bounded away from zero. Thus, $\langle L_S u_n, u_n \rangle \to -\infty$. Furthermore, from $\|u_n \|_{L^2} \leq \| w_1 \|_{L^2}$, we find that \eqref{un} holds.

Similarly, we get for $v_n$, that $(P(x) v_n'(x))' = - P(x) w_2(x) n^2 \cos(nx) + \mo(n)$ and
\begin{align}
\langle L_S v_n, v_n \rangle &= -  n^2 \int_{a_3}^{a_4} P(x) w_2^2(x) \cos^2(n x) dx + \mo(n) \\
& \geq n^2 (a_3-a_1)(a_3-a_2) \int_{a_3}^{a_4} w_2^3(x) \cos^2(nx) dx + \mo(n). \nonumber
\end{align}
Moreover,
\begin{equation*}
\int_{a_3}^{a_4} ((x-a_3)(a_4-x))^3 \cos^2(nx) dx = \frac{(a_4-a_3)^7}{280} + \mo(n^{-4}).
\end{equation*}
Therefore, $\langle L_S v_n, v_n \rangle \to +\infty$. The inequality $\| v_n \|_{L^2} \leq \| w_2 \|_{L^2}$ then implies \eqref{vn}.
\begin{theorem}
The spectrum of $L_S$ is purely discrete and accumulates at $+\infty$ and $-\infty$, i.e. the operator is unbounded from below and from above. There are no further accumulation points in the spectrum.
\end{theorem}
\begin{remark}
The singular functions $f_n$ and $g_n$ of $H_T$ are the $n$-th eigenfunctions of the operators $L_S$ and $\tilde{L}_S$, respectively. The spectra of $L_S$ and $\tilde{L}_S$ are the same, i.e.
\begin{align*}
L_S f_n &= \lambda_n f_n, \\
\tilde{L}_S g_n &= \lambda_n g_n.
\end{align*}
The above theorem states that the eigenvalues $\lambda_n$ accumulate at both $+\infty$ and $-\infty$. As a consequence (see e.g. \cite{erdelyi}, Section 4.5), when $\lambda_n$ is large and positive, the functions $f_n$ oscillate on the region where $P(x)$ is negative and decay monotonically where $P(x)$ is positive. The same is true for $g_n$. Thus, the $f_n$ are oscillatory on $(a_3,a_4)$, the $g_n$ oscillate on $(a_1,a_2)$ and they are both monotonic on $(a_2,a_3)$. The opposite is true for large negative $\lambda_n$. In this case, $f_n$ and $g_n$ both oscillate on $(a_2,a_3)$ and are monotonic outside of this interval. This corresponds to singular values $\sigma_n$ of $H_T$ close to $1$ and means that when inverting $H_T$, high frequencies of the solution can be well recovered, if they occur in the region $(a_2,a_3)$. The case $\lambda_n \to +\infty$ corresponds to $\sigma_n \to 0$. Thus, high frequencies of the solution on $(a_3,a_4)$ cannot be recovered stably. Figure \ref{fig:sing-func} shows a plot of the singular functions $f_n$ and $g_n$ for both cases.
\end{remark}

\begin{figure}[ht!]
     \begin{center}
     
        \subfigure{
            \label{fig:sing-func-0}
            \includegraphics[width=0.7\textwidth]{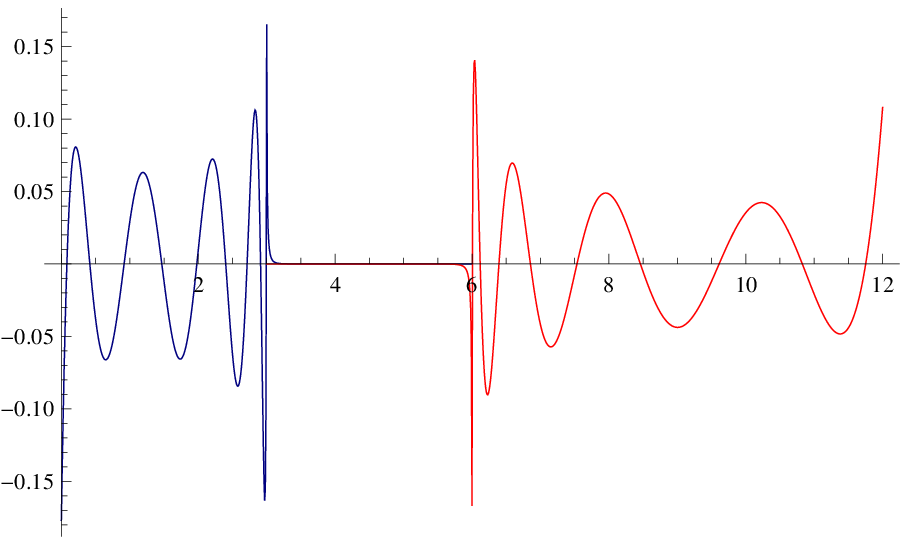}
        }\\
         \subfigure{
            \label{fig:sing-func-1}
            \includegraphics[width=0.7\textwidth]{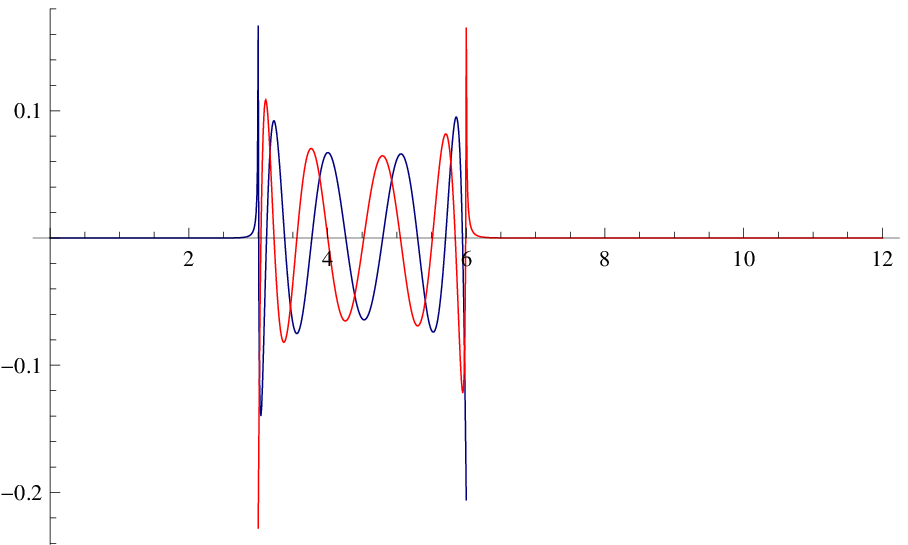}
        }
      \end{center}
    \caption{Examples of singular functions $f_n$ (red) and $g_n$ (blue) for $a_1=0, a_2=3, a_3=6, a_4=12$. \textbf{Top:} For $\sigma_n$ close to 0, the singular functions are exponentially small on $[3,6]$ and oscillate outside of $[3,6]$. \textbf{Bottom:} For $\sigma_n$ close to 1, the functions oscillate on $[3,6]$ and are exponentially small outside of the overlap region.
     }
   \label{fig:sing-func}
\end{figure}

\section{A procedure for finding the asymptotics of the singular functions}\label{sec:procedure}

We now want to study the asymptotic behavior of the eigenfunctions $f_n$ of $L_S$ and $g_n$ of $\tilde{L}_S$ as $\lambda_n \to +\infty$. In Section \ref{sigma-1} we will treat the case $\lambda_n \to -\infty$. Away from the singular points $a_i$ the solutions to the Sturm-Liouville problem for large eigenvalues are well approximated by the Wentzel-Kramers-Brillouin (WKB) method (see \cite{bender}). Close to the singularities, the solutions can be estimated by Bessel functions of the first and second kind. These two local asymptotic expansions can then be matched in the overlap of their regions of validity. This procedure was introduced for two other instances of the truncated Hilbert transform -- the interior problem and the truncated Hilbert transform \textit{with a gap} -- in \cite{k-t}, to which we refer for full details and proofs. 

\subsection{Outline of the construction of $g_n$ for $\lambda_n \to + \infty$} 
First, we start with a solution $g$ to 
\begin{equation}\label{diffeqn}
(L-\lambda) \phi = 0
\end{equation}
on $(a_1,a_2)$ and then require that it be bounded at $a_1$. We show that by analyticity, this solution extends to $\overline{\mathbb{C}} \backslash [a_2,a_4]$. Next, we extend $g$ to $(a_2,a_4)$ by analytic continuation via the upper half plane to $a_3^+$ and require $\R g(x+i0)$ to be bounded at $a_3$. With this, we can define $g$ as the function analytic on $\overline{\mathbb{C}} \backslash [a_2,a_4]$ and extended by $\R g(x+i0)$ on $(a_2,a_4)$. Then, $g$ satisfies the boundary conditions at $a_1^+$ and $a_3^-$ and we prove that it also fulfills the transmission conditions \eqref{tc1}, \eqref{tc2} at $a_2$ and hence is an eigenfunction of $\tilde{L}_S$. For large $\lambda$, the described procedure together with the local asymptotic behavior of solutions to $(L-\lambda) \phi = 0$ leads to finding the asymptotics of the eigenfunctions.

\subsection{Validity of the approach} The solution $g$ to $(L-\lambda) \phi = 0$ is bounded at $a_1$ and therefore analytic on $\mathbb{C} \backslash [a_2,a_4]$. Furthermore, $g(z) = \mathcal{O}(1/z)$ as $z \to \infty$ and $g$ is analytic at complex infinity (see \cite{k-t}). 
We want to construct a solution $g$ extended to $(a_2,a_4)$ that also satisfies the transmission conditions at $a_2$ and the boundary condition at $a_3$. This transition at $a_2$ is not analytic (see \cite{a-k}). In order to find the proper extension to $(a_2,a_4)$, we will make explicit that $g$ has to be the Hilbert transform of a function supported on $[a_2,a_4]$. To make use of this property, we first need to introduce the Riemann-Hilbert problem:

For a given function $f \in L^2(\gamma)$ on a simple smooth bounded oriented contour $\gamma \in \mathbb{C}$, find a function $F(z)$ such that
\begin{align}
&F(z) \text{ is analytic on }  \overline{\mathbb{C}} \backslash \gamma \label{an}\\
&F(z+i0) - F(z-i0) = 2 i f(z), \quad z \in \gamma \\
&F(z) \to 0 \text{ as } z \to \infty \label{anatinf}
\end{align}
This Riemann-Hilbert problem is known to have the unique solution 
\begin{equation}
F(z) = \frac{1}{\pi} \text{p.v.} \int_{\gamma} \frac{f(\tau)}{\tau-z} d\tau, \quad z \in \mathbb{C}.
\end{equation}
(see \cite{gakhov}, Sections 14.2 and 16.3). This statement can be used in a "reversed" sense: For any function $F$ analytic on $ \overline{\mathbb{C}} \backslash \gamma$ that satisfies \eqref{anatinf}, 
define the function $f$ on $\gamma$ to be
\begin{equation}
f(z) = \frac{F(z+i0)-F(z-i0)}{2i}, \quad z \in \gamma. \label{ffromF}
\end{equation}
If $f \in L^2(\gamma)$, then by construction, $F$ is the unique solution to the Riemann-Hilbert problem with right-hand side \eqref{ffromF}. Thus, $F(z) = 1/\pi \text{ p.v.} \int_{\gamma} f(\tau)/(\tau-z) d\tau$ on $\overline{\mathbb{C}} \backslash \gamma$.

Let $\gamma = [a_2,a_4]$, consider $g$ on $\overline{\mathbb{C}} \backslash [a_2,a_4]$ from above, i.e. $g$ is a solution to \eqref{diffeqn} and bounded at $a_1$, and define the function $f$ on $[a_2,a_3) \cup (a_3,a_4]$ to be 
\begin{equation}\label{f-plemelj}
f(x) := \frac{1}{2i} [g(x+i0) - g(x-i0)].
\end{equation} Clearly, $f \in L^2([a_2,a_4])$ because $g$ is analytic away from the points $a_i$ and is either bounded or has a logarithmic singularity close to the points $a_i$. With that, $F = g$ is the only solution to the corresponding Riemann-Hilbert problem by uniqueness. Let $g_{2,4}$ denote the extension of $g$ onto $(a_2,a_4)$. With a slight abuse of notation, we will denote the function $g$ extended by $g_{2,4}$ again by $g$. If we define $g_{2,4}(x) = \frac{1}{2} [g(x+i0) + g(x-i0)]$, the Plemelj-Sokhotksi formula yields that 
\begin{equation}\label{g-on-C}
g(x) \text{ extended by } g_{2,4}(x)  \quad \text{ on }  (a_2, a_4)
\end{equation}
is the Hilbert transform of $f(x)$ in \eqref{f-plemelj}, where $f$ is supported on $[a_2,a_4]$. Note that both $f(x)$ and $g_{2,4}(x)$ are solutions to $(L- \lambda) \phi = 0$, because they are linear combinations of solutions. 

For the construction of $g$ in Section \ref{sec:asymp-0}, it will be useful to express $g_{2,4}$ by the analytic continuation of $g$ via the upper half plane only, i.e. by $g(x+i0)$. This can be done as follows: Since $\lambda \in \mathbb{R}$, we can assume that $g$ is real-valued on $\mathbb{R} \backslash [a_2,a_4]$. Hence, $\I H f = 0$ on $\mathbb{R} \backslash [a_2,a_4]$ and thus, $f(x)$ is real-valued. Consequently, $g_{2,4}(x)$ is real-valued as well. If for two complex numbers $a$ and $b$, $a+b \in \mathbb{R}$ and $a-b \in \mathbb{I}$, then $\R a = \R b$ and $\I a = - \I b$. Thus, 
\begin{align}
g_{2,4}(x) &= \R g(x+i0), \label{g24}\\
f(x) &= \I g(x+i 0).
\end{align}
With these relations, we can now show that $g(x)$ in \eqref{g-on-C} satisfies the transmission conditions \eqref{tc1}, \eqref{tc2} at $a_2$.

With \eqref{dof1}, \eqref{dof2}, we can write $g$ in a neighborhood of $a_2^-$ as
\begin{align}
g(x) &=  \sum_{n=0}^{\infty} d_n (x-a_2)^n + \ln |x-a_2| \sum_{n=0}^{\infty} b_n (x-a_2)^n, \quad x<a_2 \label{g-a2}
\end{align}
Since $g$ is real-valued, $b_n, d_n \in \mathbb{R}$. 
The analytic continuation $g_c$ of $g$ from $a_2^-$ to $a_2^+$ via the upper half plane is
\begin{equation}\label{gc}
g_c(x) = \sum_{n=0}^{\infty} d_n (x-a_2)^n + (\ln |x-a_2|-i \pi) \sum_{n=0}^{\infty} b_n (x-a_2)^n, \quad x>a_2
\end{equation}
By \eqref{g-on-C}, \eqref{g24}, for $x$ to the right of $a_2$, $g$ is obtained by extracting the real part in \eqref{gc}. Comparing $ \R g_c(x)$ with \eqref{g-a2} then implies the transmission conditions \eqref{tc1} and \eqref{tc2} at $a_2$.

\begin{remark}
Another way to see that $g$ extended by $g_{2,4}$ satisfies the transmission conditions is the following: The function $f$ is a solution to $(L-\lambda) \phi = 0$ and thus is either bounded or of logarithmic singularity at $a_2$.  Suppose $f$ has a logarithmic singularity at $a_2$. Then, its Hilbert transform  will have a singularity at $a_2$ that is stronger than logarithmic. This is a contradiction to $g = H f$ being a solution to the differential equation $L g = \lambda g$. Therefore, $f$ has to be bounded at $a_2$. This implies that $g=H f$ satisfies the transmission conditions at $a_2$.
\end{remark}
The boundedness of $g_{2,4}$ at $a_3$ does not yet follow from the construction but has rather to be imposed explicitly. This is done by analytic continuation of $g$ from the interior of $(a_1,a_2)$ via the upper half plane to a neighborhood of $a_3^+$ and requiring $\R g(x+i0)$ to be bounded as $x \to a_3^+$. Using \eqref{dof1}, \eqref{dof2}, $g$ to the right of $a_3$ can be represented by
\begin{equation*}
g(x) =  \R \big[ \sum_{n=0}^{\infty} d_n (x-a_3)^n + \ln |x-a_3| \sum_{n=0}^{\infty} b_n (x-a_3)^n \big], \quad x>a_3. 
\end{equation*}
The requirement of boundedness then implies that the coefficient $b_0$ is purely imaginary. This together with using analytic continuation to express $g$ to the left of $a_3$ by
\begin{equation*}
g(x) =  \R \big[ \sum_{n=0}^{\infty} d_n (x-a_3)^n + (\ln |x-a_3|+i \pi) \sum_{n=0}^{\infty} b_n (x-a_3)^n \big] \quad x<a_3 
\end{equation*}
then yields that $g$ is also bounded at $a_3^-$.

Thus, requiring boundedness of $g$ at $a_3^+$ is sufficient to obtain that it is also bounded at $a_3^-$. This is useful because it allows for a procedure where the WKB approximation only needs to be matched to Bessel solutions on intervals where the solution is oscillatory, i.e. on $(a_1,a_2)$ and $(a_3,a_4)$.  In these intervals, we can make use of the results from \cite{k-t}, where the asymptotics of the solutions to $(L-\lambda) \phi=0$ close to the points $a_i$ were obtained in the regions where the solutions oscillate.

\section{Asymptotic analysis of the singular functions and singular values for $\sigma_n \to 0$}\label{sec:asymp-0}

In this section we want to make more precise the method motivated in the previous section. First, we need to introduce the WKB method.

As outlined in the remark at the end of Section \ref{twoacc}, for $\lambda > 0$ large, the solution $g$ to $(L-\lambda) \phi = 0$ is oscillatory where $P$ is negative, i.e. on $(a_1,a_2) \cup (a_3,a_4)$ and monotonic where $P$ is positive. We approximate the solution $g$ on $(a_1,a_2)$ away from the endpoints by the WKB method. Then, we require that $g$ be bounded at $a_1$ by matching it with a bounded local solution at $a_1^+$, which is approximated by a Bessel function of the first kind. Local solutions of \eqref{diffeqn} close to the singular points $a_i$ are approximated by linear combinations of Bessel functions of the first and second kind, \cite{k-t}. In what follows we will refer to solutions of this type as Bessel solutions. The next step is to analytically continue the WKB approximation via the upper half plane to the region to the right of $a_3$. This will be an approximation to the solution $g$ in that region because the WKB approximation is valid with a uniform accuracy (see \cite{k-t}). Recall that on $(a_2,a_4)$, $g$ is defined as $g_{2,4}(x) =  \text{Re } g(x+i0)$. At $a_3^+$, we require boundedness of $\R g(x+i0)$ by matching it with a Bessel solution of which the coefficient in front of the unbounded part is purely imaginary. As will be seen, this requirement leaves us with a discrete set $\{ \lambda_n\}_{n \in \mathbb{N}}$ for which $L_S g_n = \lambda_n g_n$, $\lambda_n \to + \infty$.

\begin{figure}[h!]
     \begin{center}
            \includegraphics[width=1\textwidth]{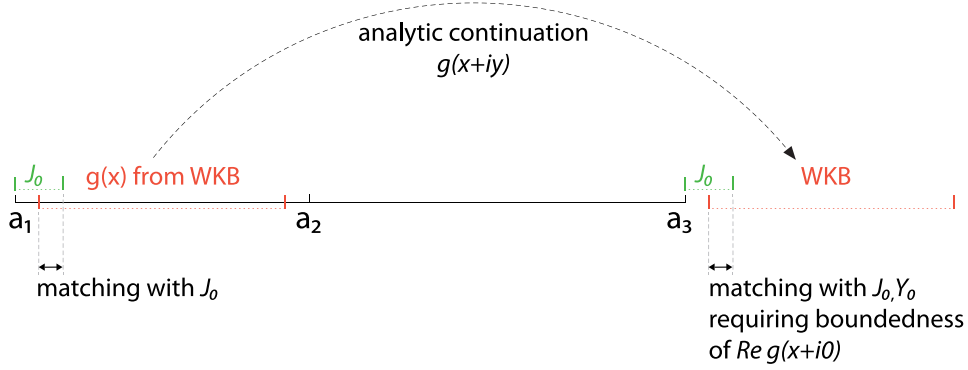}
     \end{center}
     \caption{Sketch of the construction of the $g_n$'s from WKB and Bessel approximations.}
         \label{fig:roi}
\end{figure}

In what follows we will use the following two quantities:
\begin{equation*}
K_- := \int_{a_1}^{a_2} \frac{1}{\sqrt{-P(x)}} dx,\ 
K_+ := \int_{a_2}^{a_3} \frac{1}{\sqrt{P(x)}} dx.
\end{equation*}
One can show that also
\begin{equation*}
K_- = \int_{a_3}^{a_4} \frac{1}{\sqrt{-P(x)}} dx
\end{equation*}
holds (see \cite{k-t}).
\subsection{The WKB approximation and its region of validity}
 We consider the WKB method in order to obtain approximations for solutions $g$ to $L g = \lambda g$ and large $\lambda$ on the interior of the intervals where the solutions oscillate, i.e. on $[a_1+\delta,a_2-\delta]$ and $[a_3+\delta,a_4-\delta]$ (for some small $\delta$ to be defined). We start by taking a solution on $(a_1,a_2)$ and define $\epsilon := 1/\sqrt{\lambda} $. 
Let $C_0^+$ be the upper half of the complex plane including the real line and let $a^-$ and $a_{3,4}^*$ be arbitrary but fixed real numbers for which $a^-<a_1$ and $a_{3,4}^* \in (a_3,a_4)$. It has been shown in \cite{k-t}, that for sufficiently small $\mu_1>0$ a connected region $\Lambda_- \subset C_0^+$  exists, such that 
$\Lambda_-$ contains the segment $[a^-,a_{3,4}^*]$, except for $\mo(\epsilon^{2(1-\mu_1)})$ size neighborhoods of $a_1, a_2, a_3$ and such that the following holds:
\begin{theorem}{(B.3 in \cite{k-t})}\label{thmwkb}

Using the WKB method, for every sufficiently small $\mu_1 > 0$ independent of $\epsilon$, the solutions of $(L-\lambda) \phi=0$ are linear combinations of
\begin{align}
\hat{\phi}_1(z) &= \frac{1}{P(z)^{1/4}} e^{i \sqrt{\lambda} \int_a^{z} \frac{d \xi}{\sqrt{P(\xi)}}} \big( 1+\mo(\e^{\mu_1})\big), \label{wkb1}\\
\hat{\phi}_2(z) &= \frac{1}{P(z)^{1/4}} e^{-i \sqrt{\lambda} \int_a^{z} \frac{d \xi}{\sqrt{P(\xi)}}}\big( 1+\mo(\e^{\mu_1})\big)\label{wkb2},
\end{align}
where the accuracy $\mo(\e^{\mu_1})$ is uniform in the region $\Lambda_-$. The point $a$ can, for example, be chosen to be $a_1, a_2$ or $a_3$.
\end{theorem}
The same holds in a region $\Lambda_+ \subset C_0^+$ which contains the segment $[a_{1,2}^*,a^+]$ except for $\mo(\epsilon^{2(1-\mu_1)})$ size neighborhoods of $a_2, a_3, a_4$. Here, $a^+$ and $a_{1,2}^*$ are arbitrary but fixed numbers such that $a^+>a_4$ and $a_{1,2}^* \in (a_1,a_2)$, see \cite{k-t}. Figure \ref{fig:lambda} shows a sketch of the two regions $\Lambda_-$ and $\Lambda_+$.

\begin{figure}[ht!]
     \begin{center}
     
        \subfigure{
            \label{fig:lambda-minus}
            \includegraphics[width=0.45\textwidth]{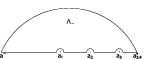}
        }
         \subfigure{
            \label{fig:lambda-plus}
            \includegraphics[width=0.45\textwidth]{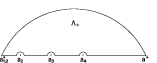}
        }
      \end{center}
    \caption{Sketches of the regions $\Lambda_-$ and $\Lambda_+$ on which the WKB approximations are valid with uniform accuracy.
     }
   \label{fig:lambda}
\end{figure}

\subsection{The Bessel solutions and their region of validity}\label{bessel-validity}
For $x \in (a_1,a_2) \cup (a_3,a_4)$ define $t = -\lambda (x-a_i) / P'(a_i)$ for fixed $i=1, \dots, 4$ and let $\mu_2$ be a small positive parameter independent of $\lambda$.

Then, the two linearly independent solutions to $(L-\lambda)\phi=0$ in a region $x-a_i = \mo(\epsilon^{2 \mu_2})$ for $t \in [0,1)$ have the properties
\begin{align}
\hat{\psi}_1(x-a_i) &= J_0(2 \sqrt{t}) + \mo(t/\lambda)  \label{J0-t-small} \\
&= J_0(2 \sqrt{t}) + \mo(\e^{2 \mu_2}) \nonumber \\
\hat{\psi}_2(x-a_i) &= Y_0(2 \sqrt{t}) + \mo(t^{1/2}/ \lambda)  \label{Y0-t-small} \\
&= Y_0(2 \sqrt{t}) + \mo(\e^{1+\mu_2}) \nonumber
\end{align}
and for $t \in [1,\mo(\epsilon^{2(\mu_2-1)})]$ 
\begin{align}
\hat{\psi}_1(x-a_i) &= J_0(2 \sqrt{t}) + t^{-1/4} \mo(\e^{2\mu_2}) \label{J0-t-large}\\
\hat{\psi}_2(x-a_i) &= Y_0(2 \sqrt{t}) + t^{-1/4} \mo(\e^{2\mu_2}) \label{Y0-t-large}
\end{align}
where $J_0$ and $Y_0$ denote the Bessel functions of the first and second kind, respectively \cite{k-t}.
\begin{lemma}[Properties of $J_0$ and $Y_0$]
The following holds for small arguments $0<z \ll 1$:
\begin{align}
J_0(z) &\to 1 \label{asymp-J0-small} \\
Y_0(z) &\sim \frac{2}{\pi} [\ln(z) + \gamma], \label{asymp-Y0-small}
\end{align}
where $\gamma$ denotes the Euler--Mascheroni constant. The asymptotic behavior for arguments  $z \to +\infty$ is
\begin{align}
J_0(z) &= \sqrt{\frac{2}{\pi z}} \big[\cos(z-\frac{\pi}{4}) + \mathcal{O}(1/z)\big] \label{asymp-J0-large} \\
Y_0(z) &= \sqrt{\frac{2}{\pi z}} \big[\sin(z-\frac{\pi}{4}) + \mathcal{O}(1/z)\big] \label{asymp-Y0-large}
\end{align}
\end{lemma}
\subsection{Overlap region of validities}

If $1-\mu_1 > \mu_2$ and $x \in (a_1,a_2) \cup (a_3,a_{3,4}^*)$, both the WKB solutions \eqref{wkb1}, \eqref{wkb2}, with accuracy $\mo(\e^{\mu_1})$, and the Bessel solutions \eqref{J0-t-large}, \eqref{Y0-t-large}, with accuracy $\mo(\e^{2 \mu_2})$, are valid in the region 
\begin{equation}\label{overlapregion}
C_1 \e^{2(1-\mu_1)} < |x-a_i| < C_2 \e^{2 \mu_2}
\end{equation}
for positive constants $C_1, C_2$ and $i=1,2,3$ (Corollary B.11, \cite{k-t}). This also holds for $x \in (a_{1,2}^*,a_2) \cup (a_3,a_4)$ and $i=2, 3, 4$.

\subsection{Derivation of the asymptotics}\label{subsec:asymptotics-of-g}

\subsubsection{The WKB approximation in $(a_1,a_2)$ away from the endpoints}

Using \eqref{wkb1} and \eqref{wkb2} with $a=a_1$, the WKB solution to $L g = \lambda g$ is 

\begin{align}
g(x) = \frac{1}{(-P(x))^{1/4}} \Big[ & \cos \Big( \frac{1}{\epsilon} \int_{a_1}^x \frac{dt}{\sqrt{-P(t)}} -\frac{\pi}{4} \Big) \cdot (1+ \mathcal{O}(\epsilon^{\mu_1})) + \label{wkb-a1} \\ 
+&c_1  \sin \Big( \frac{1}{\epsilon} \int_{a_1}^x \frac{dt}{\sqrt{-P(t)}} -\frac{\pi}{4} \Big) \cdot (1+ \mathcal{O}(\epsilon^{\mu_1})) \Big] \nonumber
\end{align}
for a constant $c_1$ and it is valid on $x \in [a_1+ \mathcal{O}(\epsilon^{2(1-\mu_1)}),a_2 - \mathcal{O}(\epsilon^{2(1-\mu_1)})]$. Here we have assumed without loss of generality that the constant in front of the cosine term is equal to $1$.

\subsubsection{Bounded Bessel solution at $a_1^+$}\label{subsec:bessel-at-a1}

Let $\hat{\psi}_1(x)$ and $\hat{\psi}_2(x)$ denote the two linearly independent solutions in the region $x-a_1 = \mo(\epsilon^{2 \mu_2})$. The boundedness of $g$ in this region requires that for $t = \lambda (a_1-x)/P'(a_1)$, $t \in [0,1)$, and constants $b_1$ and $b_2$ in
\begin{equation}
g(x) = b_1 \cdot \hat{\psi}_1(x-a_1) + b_2 \cdot \hat{\psi}_2(x-a_1),
\end{equation}
the coefficient $b_2$ be equal to zero. Thus, for $t \in [1, \mathcal{O}(\epsilon^{2(\mu_2-1)})]$
\begin{equation}\label{bdd}
g(x) = b_1 \cdot [J_0(2 \sqrt{t}) + t^{-1/4} \mathcal{O}(\epsilon^{2 \mu_2})].
\end{equation}

The two solutions \eqref{wkb-a1}, \eqref{bdd} need to be matched in the overlap region in which they are both valid, i.e. for $x$ such that
\begin{equation}\label{overlap-at-a1}
\mo(\e^{2(1-\mu_1)}) \leq x-a_1 \leq \mo(\e^{2 \mu_2}).
\end{equation}
For this, we approximate the arguments in the WKB approximation \eqref{wkb-a1}:
\begin{equation}\label{discr1}
\frac{1}{(-P(x))^{1/4}} = \frac{1+\mo(x-a_1)}{((a_1-x) P'(a_1))^{1/4}}
\end{equation}
and for the arguments in the trigonometric expressions we obtain
\begin{align}
\int_{a_1}^x \frac{dt}{\sqrt{-P(t)}} &= \int_{a_1}^x \frac{1+\mo(t-a_1)}{\sqrt{P'(a_1)(a_1-t)}} dt \\
&= \frac{1}{(-P'(a_1))^{1/2}} \int_{a_1}^x \frac{dt}{\sqrt{t-a_1}} + \int_{a_1}^x \mo \big( (t-a_1)^{1/2}\big) dt \nonumber \\
&=2 \sqrt{\frac{a_1-x}{P'(a_1)}} + \mo((x-a_1)^{3/2}). \nonumber
\end{align}
With a Taylor expansion of the cosine/sine, we then get
\begin{align}
\cos \Big( \frac{1}{\e} \int_{a_1}^x \frac{dt}{\sqrt{-P(t)}} - \frac{\pi}{4} \Big) &= \cos\Big( \frac{2}{\e} \sqrt{\frac{a_1-x}{P'(a_1)}} - \frac{\pi}{4} \Big) + \mo((x-a_1)^{3/2}/\e) \label{cos} \\ 
\sin \Big( \frac{1}{\e} \int_{a_1}^x \frac{dt}{\sqrt{-P(t)}} - \frac{\pi}{4} \Big) &= \sin \Big( \frac{2}{\e} \sqrt{\frac{a_1-x}{P'(a_1)}} - \frac{\pi}{4} \Big) + \mo((x-a_1)^{3/2}/\e) \label{sin}
\end{align}
Since $x-a_1$ lies in the overlap region \eqref{overlap-at-a1}, the following holds
\begin{equation}\label{xinoverlap}
\mo((x-a_1)^{3/2}/\e) = \mo(\e^{3 \mu_2-1}).
\end{equation}
Inserting \eqref{xinoverlap} in \eqref{cos} and \eqref{sin},
we obtain for the WKB solution in the overlap region of validity:
\begin{align}
g(x) =& \frac{1+\mo(x-a_1)}{((a_1-x) P'(a_1))^{1/4}} \cdot \Big[  \cos \Big( \frac{2}{\epsilon} \sqrt{\frac{a_1-x}{P'(a_1)}} -\frac{\pi}{4} \Big) +  \label{wkba1}  \\ 
& c_1  \sin \Big( \frac{2}{\epsilon} \sqrt{\frac{a_1-x}{P'(a_1)}}  -\frac{\pi}{4} \Big) + \mo(\e^{\text{min} \{ \mu_1, 3 \mu_2-1\}}) \Big]  \nonumber \\
=& \frac{1}{((a_1-x) P'(a_1))^{1/4}} \cdot \Big[  \cos \Big( \frac{2}{\epsilon} \sqrt{\frac{a_1-x}{P'(a_1)}} -\frac{\pi}{4} \Big) + \nonumber \\ 
&c_1  \sin \Big( \frac{2}{\epsilon} \sqrt{\frac{a_1-x}{P'(a_1)}}  -\frac{\pi}{4} \Big) + \mo(\e^{\text{min} \{ \mu_1, 3 \mu_2-1,2\mu_2\}}) \Big].\nonumber
\end{align}
We now select $\mu_1$ and $\mu_2$ such that the error term in the last equation tends to zero and such that $1-\mu_1 > \mu_2$. A convenient choice is
\begin{equation}\label{mu1mu2}
\mu_1 = \frac{1}{2}-\delta, \quad \mu_2 = \frac{1}{2}-\frac{\delta}{3}.
\end{equation}
for a small fixed $\delta >0$, as it was done in \cite{k-t}. The WKB solution has to be matched with the Bessel solution in \eqref{bdd} in the overlap region \eqref{overlap-at-a1}. We do this by matching the two solutions as $t \to \infty$ and exploiting the asymptotics \eqref{asymp-J0-large} of the Bessel function $J_0$,
which gives
\begin{equation*}
g(x) = b_1 \sqrt{\e} \Big( \frac{P'(a_1)}{a_1-x} \Big)^{1/4} \Big[ \frac{1}{\sqrt{\pi}} \cos \big( \frac{2}{\e} \sqrt{\frac{a_1-x}{P'(a_1)}} - \frac{\pi}{4} \big) + \mo\big( \frac{\e}{\sqrt{x-a_1}} \big) + \mo \big( \e^{1-2\delta/3}\big) \Big].
\end{equation*}
From $(x-a_1)^{-1/2} = \mo\big( \e^{-(\delta+1/2)}\big)$, we conclude
\begin{equation}\label{bessela1}
g(x) = b_1 \sqrt{\e} \Big( \frac{P'(a_1)}{a_1-x} \Big)^{1/4} \Big[ \frac{1}{\sqrt{\pi}} \cos \big( \frac{2}{\e} \sqrt{\frac{a_1-x}{P'(a_1)}} - \frac{\pi}{4} \big) +  \mo \big( \e^{1/2-\delta}\big) \Big].
\end{equation}
Matching the two solutions \eqref{wkba1} and \eqref{bessela1} determines the constants $b_1$ and $c_1$:
\begin{align}
b_1 &= \sqrt{\frac{\pi}{-\e P'(a_1)}} \big( 1+ \mo(\e^{1/2-\delta}) \big),\label{b1} \\
c_1 &= \mo(\e^{1/2-\delta}). \label{c1}
\end{align}
Thus, the solution $g$ is of the form
\begin{align}
g(x) =& \frac{1}{(-P(x))^{1/4}} \Big[ \cos \Big( \frac{1}{\epsilon} \int_{a_1}^x \frac{dt}{\sqrt{-P(t)}} -\frac{\pi}{4} \Big) \cdot (1+ \mathcal{O}(\epsilon^{1/2-\delta})) \label{wkba1final} \\
&+  \mathcal{O}(\epsilon^{1/2-\delta}) \sin \Big( \frac{1}{\epsilon} \int_{a_1}^x \frac{dt}{\sqrt{-P(t)}} -\frac{\pi}{4} \Big) \Big] \nonumber
\end{align}
on the interval $x \in [a_1 + \mo(\e^{1+2\delta}), a_2 - \mo(\e^{1+2\delta})]$.

\subsubsection{Analytic continuation to $a_3^+$}
The next step consists of analytically continuing $g$ in \eqref{wkba1final} to $a_3^+$ via the upper half plane. Since the WKB approximation is valid in $\Lambda_-$ with uniform accuracy $\mo(\e^{1/2-\delta})$ (Theorem \ref{thmwkb}), the analytic continuation of the WKB approximation \eqref{wkba1final} is an approximation to the analytic continuation of $g$. Taking into account the phase shifts of $P$ and using 
\begin{align}
\int_{a_1}^x \frac{dt}{\sqrt{-P(t)}} &= \int_{a_1}^{a_2} \frac{dt}{\sqrt{-P(t)}} + i \int_{a_2}^{a_3} \frac{dt}{\sqrt{P(t)}} - \int_{a_3}^x \frac{dt}{\sqrt{-P(t)}} \label{phaseshift1}  \\
&=K_- + i K_+ - \int_{a_3}^x \frac{dt}{\sqrt{-P(t)}}, \nonumber
\end{align}
we obtain
\begin{align*}
g(x+i0) = \qquad \qquad & \\
=\frac{i}{(-P(x))^{1/4}} \cdot \Big[& \cos \Big( \frac{1}{\epsilon} \int_{a_3}^x \frac{dt}{\sqrt{-P(t)}} - \frac{K_-}{\epsilon} - i \frac{K_+}{\epsilon} + \frac{\pi}{4}  \Big) \cdot \big( 1+ \mo(\e^{1/2-\delta}) \big) \\
&- \sin \Big(  \frac{1}{\epsilon} \int_{a_3}^x \frac{dt}{\sqrt{-P(t)}} - \frac{K_-}{\epsilon} - i \frac{K_+}{\epsilon} + \frac{\pi}{4}  \Big) \cdot \mo(\e^{1/2-\delta}) \Big], 
\end{align*}
where $x \in [a_3+\mo(\e^{1+2\delta}),a^*_{3,4}]$.  
The properties of the complex valued trigonometric functions yield
\begin{align*}
g(x+i0) =& \frac{i}{(-P(x))^{1/4}} \cdot \Big[ \Big\{ \cos \Big( \frac{1}{\epsilon} \int_{a_3}^x \frac{dt}{\sqrt{-P(t)}} - \frac{K_-}{\epsilon} + \frac{\pi}{4} \Big) \cdot \cosh \Big(  - \frac{K_+}{\epsilon}   \Big) \\
&-i \sin \Big( \frac{1}{\epsilon} \int_{a_3}^x \frac{dt}{\sqrt{-P(t)}} - \frac{K_-}{\epsilon} + \frac{\pi}{4} \Big) \cdot \sinh \Big(  - \frac{K_+}{\epsilon}   \Big) \Big\}  \cdot \big( 1+ \mo(\e^{1/2-\delta}) \big) \\
&+ \mo \big( \e^{1/2-\delta} \big) \cdot \Big\{ \sin \Big(  \frac{1}{\e} \int_{a_3}^x \frac{dt}{\sqrt{-P(t)}} - \frac{K_-}{\e} + \frac{\pi}{4} \Big) \cdot \cosh \Big( - \frac{K_+}{\e} \Big) \\
&+ i \cos \Big(\frac{1}{\epsilon} \int_{a_3}^x \frac{dt}{\sqrt{-P(t)}} - \frac{K_-}{\e} + \frac{\pi}{4} \Big) \cdot \sinh \Big( - \frac{K_+}{\e} \Big) \Big\} \Big].
\end{align*}
So far, $g$ is a function that is not normalized on $[a_1,a_3]$. However, we will need to work with singular functions that have their $L^2$-norm equal to $1$ on $[a_1,a_3]$ in order to estimate the singular values correctly. Thus, we incorporate $\| g\|_{L^2([a_1,a_3])}$ derived in \eqref{normg} in the Appendix, simplify the above expression and use the relation $\sin x = \cos(x-\frac{\pi}{2})$, to obtain a new normalized function $g$:
\begin{align}
g(x+i0) = \sqrt{\frac{2}{K_-}}\frac{- e^{K_+/\e}}{2 (-P(x))^{1/4}} \cdot \Big[ &\cos \big( \frac{1}{\e} \int_{a_3}^x \frac{dt}{\sqrt{-P(t)}} - \frac{K_-}{\e} - \frac{\pi}{4} \big) \label{wkba3plus}  \\
+ i  \sin \big( \frac{1}{\e} & \int_{a_3}^x \frac{dt}{\sqrt{-P(t)}} - \frac{K_-}{\e} - \frac{\pi}{4} \big) + \mo \big( \e^{1/2-\delta} \big) \Big] \nonumber
\end{align}

Next, we match this solution to a linear combination of Bessel approximations at $a_3^+$ and then require boundedness of its real part. 

In the overlap region \eqref{overlapregion} close to $a_3^+$ where both the WKB and the Bessel solution are valid, we define $t = \lambda (a_3-x)/P'(a_3)$. The function $P$ in \eqref{wkba3plus} can be approximated in the same way as it was done at $a_1^+$ in \eqref{discr1} -- \eqref{sin}. Factorizing the trigonometric expression, the WKB solution \eqref{wkba3plus} can then be written as
\begin{align}
g(x+i0) = &- \sqrt{\frac{1}{2 K_-}}\frac{ e^{K_+/\e}}{((a_3-x) P'(a_3))^{1/4}} e^{-i K_-/\e} \cdot \big[ \cos(\frac{2}{\e} \sqrt{\frac{a_3-x}{P'(a_3)}}-\frac{\pi}{4})  \label{wkba3-overlap} \\
& + i \sin (\frac{2}{\e} \sqrt{\frac{a_3-x}{P'(a_3)}}-\frac{\pi}{4}) + \mo(\e^{1/2-\delta})\big]. \nonumber
\end{align}
In the overlap region \eqref{overlapregion} the Bessel solution is a linear combination of the solution $\hat{\psi}_1$ in \eqref{J0-t-large} bounded at $a_3^+$ and the solution $\hat{\psi}_2$ in \eqref{Y0-t-large} having a singularity at $a_3^+ $:
\begin{equation}\label{bessela3-both}
g(x) = b_3 \hat{\psi}_1(x-a_3) + c_3 \hat{\psi}_2(x-a_3)
\end{equation}
for constants $b_3$ and $c_3$.  To ensure boundedness at $a_3^+$ of the real part of \eqref{bessela3-both}, we will need to impose $\R c_3 = 0$. The asymptotics of $J_0$ and $Y_0$ in \eqref{asymp-J0-large} and \eqref{asymp-Y0-large} for large $t$ allow to write the Bessel solution similarly to \eqref{bessela1} but with an additional term in $Y_0$:
\begin{align}
g(x) =  \sqrt{\frac{\e}{\pi}} \Big( \frac{P'(a_3)}{a_3-x}\Big)^{1/4} \Big( &b_3 \big[ \cos (\frac{2}{\e} \sqrt{\frac{a_3-x}{P'(a_3)}}-\frac{\pi}{4}) + \mo(\e^{1/2-\delta})\big] + \label{bessela3-overlap} \\
+ &c_3 \big[ \sin (\frac{2}{\e} \sqrt{\frac{a_3-x}{P'(a_3)}}-\frac{\pi}{4}) + \mo(\e^{1/2-\delta})\big] \Big). \nonumber
\end{align}
From the matching of \eqref{wkba3-overlap} with \eqref{bessela3-overlap} for $t \to \infty$, we obtain
\begin{equation}\label{b3c3}
b_3 = - i c_3 (1+\mo(\e^{1/2-\delta})).
\end{equation}
The requirement $\R c_3 = 0$ then implies $\I b_3 = \I c_3 \cdot \mo(\e^{1/2-\delta})$ and hence 
\begin{equation}
b_3 = \R b_3 \cdot (1+\mo(\e^{1/2-\delta}))
\end{equation}
or more explicitly
\begin{equation}\label{b3-real}
b_3 = - \sqrt{\frac{\pi}{2 K_- \e}} \frac{e^{K_+/\e}}{\sqrt{-P'(a_3)}} \cos \big(\frac{K_-}{\e} \big)(1+ \mo(\e^{1/2-\delta})).
\end{equation}
The matching also yields that $\R c_3 = 0$ implies $\R (i e^{-iK_-/\epsilon}) = \mo(\epsilon^{1/2-\delta})$. Thus,
\begin{equation}
\sin \big(\frac{K_-}{\epsilon}\big) = \mo(\epsilon^{1/2-\delta})
\end{equation}
and as a result
\begin{equation}
\frac{K_-}{\e} = n \pi + \mo \big( \e^{1/2-\delta} \big).\label{discrete-eps}
\end{equation}
for $n \in \mathbb{N}$. This equation for the parameter $\epsilon = 1/\sqrt{\lambda}$, where $\lambda$ is a large positive eigenvalue of the operator $L_S$, shows the essential property of the spectrum of $L_S$ to be purely discrete and, in addition, reveals the rate at which the eigenvalues tend to $+\infty$. Since the spectrum of $L_S$ is both unbounded above and below, we have to make a choice in terms of the enumeration of the eigenvalues $\lambda_n$. Equation \eqref{discrete-eps} shows that we can choose the enumeration such that
\begin{equation}
\sqrt{\lambda_n} = \frac{n \pi}{K_-} + \mo \big( n^{-1/2+\delta} \big), \quad n \in \mathbb{N} \label{discrete-lambda}
\end{equation}
holds.
With this and \eqref{b3-real} we finally obtain the coefficient $b_3$:
\begin{equation}\label{b3}
b_3 = (-1)^{n+1} \sqrt{\frac{\pi}{2 K_- \e}} \frac{e^{K_+/\e}}{\sqrt{-P'(a_3)}} (1+ \mo(\e^{1/2-\delta})).
\end{equation}

\begin{figure}[h!]
     \begin{center}
            \includegraphics[width=0.9\textwidth]{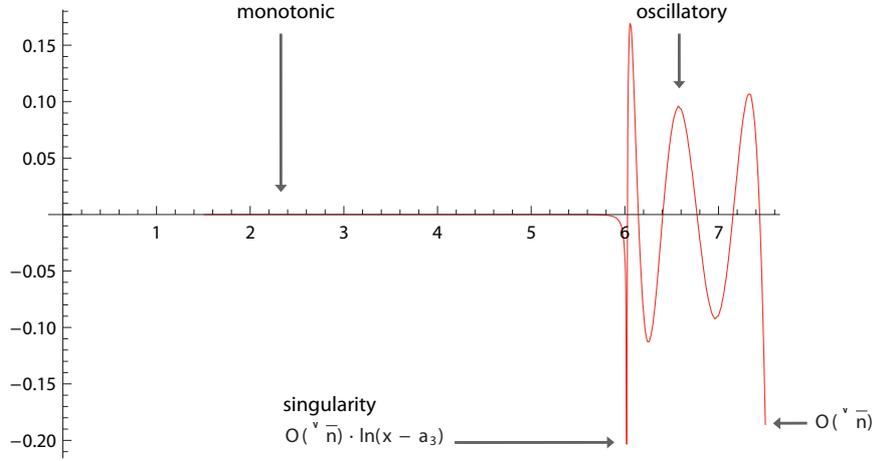}
     \end{center}
     \caption{The asymptotic behavior of the singular functions $f_n$ as $\sigma_n \to 0$.}
         \label{fig:asymp-sf}
\end{figure}

\subsection{Asymptotic behavior of the singular values accumulating at zero}
In the previous sections we have obtained the asymptotics of the functions $g$ with $\| g\|_{L^2([a_1,a_3])}=1$ and the property that $\chi_{[a_1,a_3]} g$ are the singular functions of $H_T$ for singular values close to zero. We found $g$ by defining it to be equal to the analytic function on $\bar{\mathbb{C}} \backslash [a_2,a_4]$ extended by  $g_{2,4}$ on $(a_2,a_4)$, see \eqref{g-on-C}.
These functions $g$ are the Hilbert transforms of functions $f$ that are supported on $[a_2,a_4]$. If we normalize $f$ as well, this reads $H f = \sigma g$, where $\sigma \ll 1$ is the corresponding singular value of $H_T$. Applying the Hilbert transform on both sides gives $H g = -\frac{1}{\sigma} f$. Thus, in order to estimate $\sigma$, we can proceed as follows: 
\begin{enumerate}
\item Estimate the jump discontinuity $g(a_3^+)-g(a_3^-)$ \\
\item Find the logarithmic term in $(H g)(a_3^+)$ \\
\item Determine the logarithmic term in $f(a_3^+)$ \\
\item Estimate $\sigma = - f(a_3^+)/(H g)(a_3^+)$
\end{enumerate}

Combining the asymptotics of the Bessel solutions \eqref{J0-t-small}, \eqref{Y0-t-small} with the representation \eqref{bessela3-both} of $g$ close to $a_3^+$, yields the following asymptotics for $g$:
\begin{align*}
g(x) &= b_3 \cdot \big[ J_0(2 \sqrt{t}) + \mo(\e^{1-2\delta/3}) \big] + c_3 \cdot \big[ Y_0(2 \sqrt{t}) + \mo(\e^{3/2-\delta/3}) \big]
\end{align*}
where $t = \lambda (a_3-x)/P'(a_3) \in [0,1)$. Using the relation \eqref{b3c3} between $b_3$ and $c_3$, we can further write this as 
\begin{equation*}
g(x)= -i c_3 \cdot \big[ J_0(2 \sqrt{t}) + \mo(\e^{1/2-\delta}) \big] + c_3 \cdot \big[ Y_0(2 \sqrt{t}) + \mo(\e^{3/2-\delta/3}) \big]
\end{equation*}
On the other hand, we know from the theory of Fuchs-Frobenius that close to $a_3^+$ a solution to $(L- \lambda) g = 0$ is of the form
 \begin{equation*}
g(x) =  \phi_1(x) + \phi_2(x) \ln |x-a_3|
\end{equation*}
where $\phi_1, \phi_2$ are analytic close to $a_3$. The requirement that $\R g (a_3^+)$ be bounded implies $\R \phi_2(a_3) = 0$. The analytic continuation of $g$ to a neighborhood of $a_3^-$ is given by
\begin{equation*}
g_{c}(x) =  \phi_{1}(x) + \phi_{2} (x) [ \ln |x-a_3|+i \pi].
\end{equation*}
According to \eqref{g24}, $g$ at $a_3^-$ is equal to $\R g_{c}(x)$. This determines the jump discontinuity of $g$ across $a_3$ to be $- i \pi \phi_2(a_3)$. Hence, using the asymptotics \eqref{asymp-Y0-small} of $Y_0$,  
the jump discontinuity of $g$ at $a_3$ is equal to
\begin{equation*}
g(a_3^+) - g(a_3^-) = - i \pi \phi_2(a_3) = -i c_3 = b_3 (1+ \mo(\e^{1/2-\delta})).
\end{equation*}

This allows to estimate the logarithmic term in $H g$ to be $- \frac{1}{\pi} b_3 (1+ \mo(\e^{1/2-\delta})) \ln|x-a_3|$ (see Section 8.2 in \cite{gakhov}). 

Next, we find $f$, such that supp $f = [a_2,a_4]$ and $L f = \lambda f$ with a WKB approximation which holds on the region $\Lambda_+$. On $(a_3,a_4)$, $f$ is oscillatory, so analogously to the procedure for $g$, we start with the WKB approximation on $[a_3+\mo(\e^{1+2 \delta}),a_4-\mo(\e^{1+2 \delta})]$ and require boundedness at $a_4^+$. This determines $f$ (similarly to \eqref{wkba1final} for $g$) up to a constant:
\begin{align*}
f(x) = \frac{1}{(-P(x))^{1/4}} \Big[ &\cos \Big( \frac{1}{\epsilon} \int_{x}^{a_4} \frac{dt}{\sqrt{-P(t)}} -\frac{\pi}{4} \Big) \cdot (1+ \mathcal{O}(\epsilon^{1/2-\delta})) \\
&+ \sin \Big( \frac{1}{\epsilon} \int_{x}^{a_4} \frac{dt}{\sqrt{-P(t)}} -\frac{\pi}{4} \Big) \cdot  \mathcal{O}(\epsilon^{1/2-\delta}) \Big]. 
\end{align*}
Before, we estimated $g$ at $a_3^+$ and required its real part to be bounded. Now, in the procedure for $f$, we are interested in estimating the \textit{unbounded} part of $f$ at $a_3^+$.
We make use of the relation
\begin{align*}
\int_{a_3}^x \frac{dt}{\sqrt{-P(t)}} = \int_{a_3}^{a_4} \frac{dt}{\sqrt{-P(t)}} + \int_{a_4}^x \frac{dt}{\sqrt{-P(t)}} = - \int_x^{a_4} \frac{dt}{\sqrt{-P(t)}} + K_-,
\end{align*}
that allows to rewrite $f$ on $[a_3+\mo(\e^{1+2 \delta}),a_4-\mo(\e^{1+2 \delta})]$:
\begin{align}
f(x) = \frac{1}{(-P(x))^{1/4}} \Big[& \cos \Big( - \frac{1}{\epsilon} \int_{a_3}^{x} \frac{dt}{\sqrt{-P(t)}} + \frac{K_-}{\e}-\frac{\pi}{4} \Big) \cdot (1+ \mathcal{O}(\epsilon^{1/2-\delta})) \nonumber \\
&+  \mathcal{O}(\epsilon^{1/2-\delta}) \sin \Big( - \frac{1}{\epsilon} \int_{a_3}^{x} \frac{dt}{\sqrt{-P(t)}} +  \frac{K_-}{\e} -\frac{\pi}{4} \Big) \Big]. 
\end{align}
Using \eqref{discrete-eps} and trigonometric identities, it then follows that
\begin{align}
f(x) = \frac{(-1)^n}{(-P(x))^{1/4}} \Big[& - \sin \Big( \frac{1}{\epsilon} \int_{a_3}^{x} \frac{dt}{\sqrt{-P(t)}} -\frac{\pi}{4} \Big) + \mathcal{O}(\epsilon^{1/2-\delta}) \label{wkb-f-a3}  \\
&-  \cos \Big( \frac{1}{\epsilon} \int_{a_3}^{x} \frac{dt}{\sqrt{-P(t)}} -\frac{\pi}{4} \Big) \cdot \mathcal{O}(\epsilon^{1/2-\delta}) \Big] \nonumber
\end{align}
on $[a_3+\mo(\e^{1+2 \delta}),a_4-\mo(\e^{1+2 \delta})]$. 

In a neighborhood of $a_3^+$, $f$ can be represented as a linear combination of the Bessel solutions \eqref{J0-t-large} and \eqref{Y0-t-large}. For constants $b'_3$, $c'_3$,
\begin{equation}
f(x) = b_3' \big[ J_0(2 \sqrt{t}) + t^{-1/4} \mo(\e^{1-2 \delta/3}) \big] + c_3' \big[ Y_0(2 \sqrt{t}) + t^{-1/4} \mo(\e^{1-2 \delta/3}) \big],
\end{equation}
where $t = \lambda (a_3-x)/P'(a_3)$ and $t \in [1,\mo(\epsilon^{-1-2\delta/3})]$. Using the asymptotics of the Bessel functions for $t \to +\infty$ \eqref{J0-t-large}, \eqref{Y0-t-large} to match the above with the WKB solution \eqref{wkb-f-a3} in their overlap region of validity (similarly as in Section \ref{subsec:bessel-at-a1}) gives
\begin{align}
b_3' &= (-1)^{n+1}\sqrt{\frac{\pi}{-\e P'(a_3)}} \cdot \mo(\e^{1/2-\delta}), \label{b3-prime}\\
c_3'  &= (-1)^{n+1}\sqrt{\frac{\pi}{-\e P'(a_3)}} \cdot (1+ \mo(\e^{1/2-\delta})). \label{c3-prime}
\end{align}
After normalization of $f$ (as it was done for $g$, see Appendix), we can use \eqref{asymp-Y0-small} to find the logarithmic term in $f(a_3^+)$ (up to a sign):
\begin{displaymath}
\frac{(-1)^{n+1}}{\pi} \sqrt{\frac{2 \pi}{-\e P'(a_3) K_-}}  \ln|x-a_3| (1+ \mo(\e^{1/2-\delta})).
\end{displaymath}
The sign of $f$ is then determined by $f = - \sigma H g$ and $\sigma >0$. This yields
\begin{align*}
(Hg)(a_3^+)/f(a_3^+) &= \frac{ \frac{(-1)^n}{\pi}  \sqrt{\frac{\pi}{2 K_- \e}} \frac{e^{K_+/\e}}{\sqrt{-P'(a_3)}}}{(-1)^{n+1}\sqrt{\frac{2 \pi}{-\e P'(a_3) K_-}} \cdot \frac{1}{\pi}} \cdot (1+ \mo(\e^{1/2-\delta}))  \\
&= - \frac{1}{2} e^{K_+/\e} (1+ \mo(\e^{1/2-\delta}))
\end{align*}
and
\begin{align}
\sigma &= 2 e^{-K_+/\e} (1+ \mo(\e^{1/2-\delta}))
\end{align}
\begin{theorem}\label{thm-sigma}
Let $\lambda_n$ be enumerated as in \eqref{discrete-lambda}. Then, the singular values $\sigma_n$ of $H_T$ that accumulate at zero, behave asymptotically like
\begin{equation}
\sigma_n = 2 e^{-n \pi K_+/K_-} (1+ \mo(n^{-1/2+\delta})),\ n\to\infty.
\end{equation}
\end{theorem}
This result shows the \textit{severe} ill-posedness of the underlying problem of reconstructing a function $f$ from $H_T f = g$ for given $g$: A subsequence of the singular values $\sigma_n$ of $H_T$ decays to zero, resulting in the unboundedness of the inverse of $H_T$. As a consequence, small perturbations in $g$ due to measurement noise will result in unreliable predictions for $f$. Unlike in cases of so-called \textit{mild} ill-posedness, where the singular values decay to zero at a polynomial rate, the singular values $\sigma_n$ of $H_T$ decay to zero exponentially, resulting in \textit{severe} ill-posedness.

\begin{remark}
The most natural way to find the asymptotics of $\sigma_n$ would be to estimate the jump discontinuity of the singular functions $\chi_{[a_1,a_3]} g$ and then use $H_T^* g = \sigma f$. However, the jump discontinuity of $\chi_{[a_1,a_3]} g$ at $a_3^-$ can only be estimated to be of the order $b_3 \cdot \mo(\e^{1-2\delta/3})$, where $b_3$ (see (\ref{b3})) contains the term $e^{K_+/\e}$. Thus, the coefficient in front of the logarithmic term in $H_T^* g$ will also be of the order $b_3 \cdot \mo(\e^{1-2\delta/3})$, which results in the useless estimate $\sigma = \mo(e^{K_+/\e} \cdot \e^{1-2\delta/3})$. Therefore, it was necessary to replace $H_T$ by the full Hilbert transform $H$ and to consider $H g = - \frac{1}{\sigma} f$ instead of $H_T^* g = \sigma f$  
to obtain the result of Theorem \ref{thm-sigma}.
\end{remark}

\section{Asymptotic analysis for the case of $\sigma_n \to 1$}\label{sigma-1}

The previous section described how to derive the asymptotic behavior of the singular values in the neighborhood of their accumulation point at zero.

Here we show how to easily obtain the asymptotic behavior around the second accumulation point equal to $1$ using a symmetry property that allows to exploit the analysis done for the first accumulation point. 

We define the operator $H_{T,c} := \mathp_{[a_2,a_4]} H \mathp_{([a_1,a_3])^c}$,  where $(\cdot)^c$ denotes the complement in $\mathbb{R}$ and $\mathp$ is the projection operator defined in Section \ref{sec:intro}. Without loss of generality we assume $a_1 < 0 < a_2 < a_3 < a_4$. Consider a singular function $f \in L^2([a_2,a_4])$ of $H_T$ with singular value $\sigma$. As it was shown in \cite{a-k}, the spectrum of $H_T^* H_T$ is bounded above by 1. Therefore we can define 
$\beta^2=1-\sigma^2$, and see that $f$ satisfies the eigenequation
\begin{equation}
f - \beta^2 f =  H_{T}^* \,  H_{T} f.
\end{equation}
On the other hand we have $H^* H = I$, where $H$ is the full Hilbert transform on the line. Hence $f$ also satisfies \begin{equation}
f  = H_T^* \,  H_{T} f + H_{T,c}^* \,  H_{T,c} f.
\end{equation}
Subtracting the two equations we obtain a new eigenequation for $f$, now with eigenvalue $\beta^2$:
\begin{equation}
\beta^2 f = H_{T,c}^* \,  H_{T,c} f. 
\end{equation}
We will relate this eigenequation to an eigenequation for a different truncated Hilbert problem, obtained by the transformation  $x \leftrightarrow 1/x$.
Define $\eta=1/x$ and the singular points $\eta_j=1/a_j, j=1,\dots,4$. These are ordered as $\eta_1 < 0 < \eta_4 < \eta_3 < \eta_2$. Furthermore, we define the function
${\bar f}(\eta) = \eta^{-1} f(\eta^{-1})$. Note that the support of $\bar f$ is $\eta_4 < \eta < \eta_2$. With these notations, we have, noting that $0 \notin (a_2,a_4)$,
\begin{eqnarray}
 x \left ( H_{T,c} f \right )(x) &=& x \, \frac{1}{\pi} \, \text{p.v.}\int_{a_2}^{a_4}  \frac{f(y)}{y-x} dy= \frac{1}{\pi} \, \text{p.v.}\int_{a_2}^{a_4}  \frac{y \, f(y)}{(1/x-1/y)} \frac{dy}{y^2} \label{trafo} \\
&=& - \frac{1}{\pi} \,  \text{p.v.}\int_{\eta_4}^{\eta_2} \, \frac{{\bar f}(\eta)}{\eta-\xi}  d \eta = - (\bar{H}_{T} {\bar f})(\xi) \ \mbox{  with  } \xi=1/x,  \nonumber 
\label{xHfx}
\end{eqnarray}
where we define the operator $\bar{H}_{T}: L^2([\eta_4,\eta_2]) \rightarrow L^2([\eta_1,\eta_3])$\footnote{In \eqref{trafo} we have assumed 
that the variable transformation in the principal value integrals can be handled in the same way as for ordinary integrals. For a proof of this property we refer to \cite{gakhov}, Section 3.5.
} to be:
\begin{equation}
(\bar{H}_{T} h)(\xi) = \frac{1}{\pi} \,  \text{p.v.} \int_{\eta_4}^{\eta_2}  \, \frac{h(\eta)}{\eta-\xi} d \eta.
\end{equation}
The range in $\xi$ is obtained from
\begin{equation}
x \in ([a_1,a_3])^c = (-\infty,a_1) \cup (a_3, \infty) \Rightarrow \xi \in (\eta_1,0) \cup (0, \eta_3)=(\eta_1,\eta_3)
\end{equation}
We now apply the adjoint transform, and calculate for $a_2 < z < a_4$:
\begin{eqnarray*}
z \, \left (  H_{T,c}^* \,  H_{T,c} f \right)(z) &=& \frac{z}{\pi} \left \{ \int_{-\infty}^{a_1 < 0} + \int^{\infty}_{a_3 > 0} \right \}  \, \frac{(H_{T,c} f )(x)}{z-x}dx \\
&=& \frac{1}{\pi \omega}  \left \{ \int_{\eta_1}^0 + \int_0^{\eta_3 } \right \}  \, \frac{(H_{T,c} f )(1/\xi)}{(1/\omega-1/\xi)} \frac{d \xi}{\xi^2} \\
&=& -\frac{1}{\pi} \int_{\eta_1}^{\eta_3}  \, \frac{(1/\xi) \, (H_{T,c} f )(1/\xi)}{(\omega-\xi)} d \xi \\
&=&  \frac{1}{\pi} \int_{\eta_1}^{\eta_3}  \, \frac{(\bar{H}_T{\bar f})(\xi)}{(\omega-\xi)} d \xi=\bar{H}_T^*  \bar{H}_T {\bar f}(\omega)  
\end{eqnarray*}
where $\omega=1/z$.  We conclude that $\bar{H}_T^*  \bar{H}_T {\bar f} = \beta^2 {\bar f}$, hence $\beta^2$ is an eigenvalue for the truncated Hilbert problem defined by
 $\eta_1 < \eta_4 < \eta_3 < \eta_2$. 
 
 The implication of this result for the asymptotic behavior of the singular values around the accumulation points 0 and 1 are as follows.
 Consider the case $\beta^2 \rightarrow 0$.
 From the previous section we know that the asymptotic behavior of these eigenvalues (which are the squares of the singular values of $\bar{H}_T$) is given by
 \begin{equation}
 \beta_n = 2 e^{- n \pi \bar{K}_+/\bar{K}_-} (1+ \mo(n^{-1/2+\delta})),
 \end{equation}
 with 
 \begin{eqnarray*}
 {\bar K}_+  &=&  \int_{\eta_4}^{\eta_3}  \left \{ (t-\eta_1)(t-\eta_2)(t-\eta_3)(t-\eta_4) \right \}^{-1/2}dt = (|a_1| a_2 a_3 a_4)^{1/2} \, K_-, \\
  {\bar K}_-  &=&  \int_{\eta_3}^{\eta_2}  \left \{ -(t-\eta_1)(t-\eta_2)(t-\eta_3)(t-\eta_4) \right \}^{-1/2} dt= (|a_1| a_2 a_3 a_4)^{1/2} \, K_+,
 \end{eqnarray*}
where the last equalities can be checked by substituting $t=1/y$ in the integrals.
 Using the previous result and recalling the definition $\beta^2=1-\sigma^2$, we obtain the asymptotic behavior in the neighborhood of $1$ of the singular values of the original problem defined by $a_1,a_2,a_3,a_4$:
\begin{theorem}\label{thm-sigma-1}
The singular values $\sigma_{-n}$, $n \in \mathbb{N}$, accumulating at $1$ have the following asymptotic behavior
 \begin{equation}
 \sigma_{-n} = \sqrt{1-\beta^2_n} =\sqrt{ 1- 4 e^{-2 n \pi K_-/K_+} (1+ \mo(n^{-1/2+\delta}))}.
 \label{slope1}
\end{equation}
\end{theorem}
\section{Comparison of numerics and asymptotics}\label{sec:numerics}
In the previous sections, the asymptotic behavior of the singular value decomposition has been derived. Although these asymptotics only hold in the limit $n \to \infty$, we would like to illustrate that they also yield a good approximation of the SVD for small $n$. For this, we compare the SVD from the asymptotic formulas with the SVD of a discretized version of the operator $H_T$.

For our example, we choose the points $a_i$ to be $a_1=0,  a_2 = 3, a_3 = 6, a_4 = 12$  
and the discretization $\mathbf{H_T}$ of $H_T$ to be a uniform sampling with $601$ partition points in the interval $[0,6]$ and $901$ points in $[3,12]$. Let vectors $X$ and $Y$ denote the partition points of $[0,6]$ and $[3,12]$ respectively. To overcome the singularity of the Hilbert kernel the vector $X$ is shifted by half of the sample size. The $i$-th components of the two vectors $X$ and $Y$ are given by $X_i = \frac{1}{100} (i + \frac{1}{2})$ and $Y_i = 3 + \frac{1}{100} i$; $H_T$ is then discretized as $(\mathbf{H_T)_{i,j}} = (1/\pi) (X_i-Y_j)$, $\mathbf{i}=0, \dots,600$, $\mathbf{j} = 0, \dots, 900$. 

Let $\mathbf{s_i}$, $\mathbf{i}=0, \dots, 313$ denote the non-zero singular values of the matrix $\mathbf{H_T}$. Table \ref{tab:sv} shows a list of a few singular values indicating that for $\mathbf{i} = 0, \dots, 300$ the values $\mathbf{s_i}$ are close to $1$, whereas they are close to $0$ for $\mathbf{i} = 302, \dots, 313$. Although in theory, $0$ itself is not a singular value of $H_T$ but the singular values only decay to $0$, they do this at an exponential rate. In practice, this leads to matrix realizations of $H_T$ which effectively have a large nullspace.

\begin{table}[h]
\begin{center}
    \begin{tabular}{ | l | l | l | l | l | l | l |}
    \hline
    $\mathbf{i}$ & $300$ & $301$ & $302$ & $303$ & $304$ & $305$ \\ \hline
    $\mathbf{s_i}$ & $0.9942962$ & $0.6630176$ & $0.0397114$ & $1.1321 \cdot 10^{-3}$ & $2.9846 \cdot 10^{-5}$  & $7.7106 \cdot 10^{-7}$\\ 
\hline    

    \end{tabular}
\
\caption{The singular values of $\mathbf{H_T}$ in the transition from $1$ to $0$.}\label{tab:sv}
\end{center}
\end{table}

We compare the singular values $\mathbf{s_i}$, $\mathbf{i} = 302, \dots, 313$ of $\mathbf{H_T}$ with the asymptotic behavior of the singular values $\sigma_n$ of $H_T$ for $\sigma_n \to 0$ (see Theorem \ref{thm-sigma}). Here, we neglect the error terms, i.e. we consider the asymptotic form $\sigma_n \approx 2 e^{-n \pi K_+/K_-}$, for $n=1, \dots, 12$. Finding the set of indices for $n$ that match the chosen indices $\mathbf{i}=302, \dots, 313$ is done by hand. Figure \ref{fig:sv-comp-0} shows a logarithmic plot of this comparison. While Theorem \ref{thm-sigma} only guarantees that $2 e^{-n \pi K_+/K_-}$ is a good approximation of the singular values $\sigma_n$ for $n \to \infty$, our example demonstrates good alignment already for $n=1$.

\begin{figure}[h!]
     \begin{center}
            \includegraphics[width=0.7\textwidth]{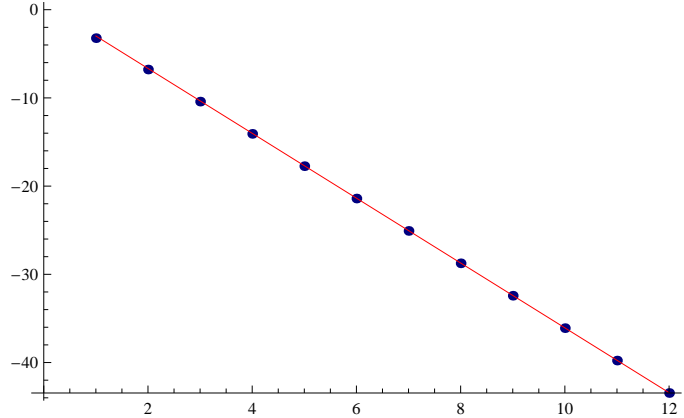}
     \end{center}
     \caption{Logarithmic plot of the asymptotic (red line) and numerical values (blue dots) of the singular values tending to zero.}
         \label{fig:sv-comp-0}
\end{figure}

Similarly, we perform a comparison of the singular values $\mathbf{s_i}$, $\mathbf{i} = 293, \dots, 300$ of $\mathbf{H_T}$ with the result from Theorem \ref{thm-sigma-1} on the asymptotic behavior of the singular values $\sigma_{-n} \to 1$. Again, the error terms are neglected, so that $\sigma_{-n} \approx \sqrt{ 1- 4 e^{-2 n \pi K_-/K_+}}$ for $n=1, \dots, 8$ is considered instead. A plot comparing $\log(1-\mathbf{s_i}^2)$ with $\log(4 e^{-2 n \pi K_-/K_+})$ is shown in Figure \ref{fig:sv-comp-1}, illustrating the good alignment for small values of $n$.

\begin{figure}[h!]
     \begin{center}
            \includegraphics[width=0.7\textwidth]{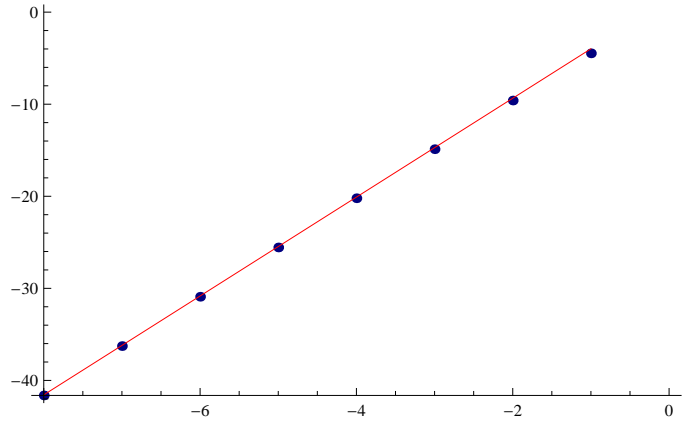}
     \end{center}
     \caption{Logarithmic plot of the asymptotic (red line) and numerical values (blue dots) of $1-\sigma_{-n}^2$ for the singular values $\sigma_{-n}$ tending to $1$.}
         \label{fig:sv-comp-1}
\end{figure}

To conclude the numerical illustration, we compare the singular vector $\mathbf{g_{307}}$ of $\mathbf{H_T}$ with the asymptotic behavior obtained for the singular function $g_6$ of $H_T$.
For this again, only the leading terms in the asymptotic expansions are taken into consideration. To define the approximation to $g_6$ on the entire interval $[0,6]$, we first consider the plots of the WKB and Bessel approximations close to a point $a_i$. Then, the point of transition from the Bessel to the WKB approximation is set by hand at a point of good alignment between the two functions. Figure \ref{fig:sf-comp} shows the approximation to $g_6$ obtained from the asymptotics compared to the singular vector $\mathbf{g_{307}}$. In Figure \ref{fig:sf-comp-log}, a logarithmic plot indicates that the asymptotic form is a very good approximation to $\mathbf{g_{307}}$ also on the region where it decays, i.e. on $[3,6]$.

\begin{figure}[h!]
     \begin{center}
            \includegraphics[width=0.7\textwidth]{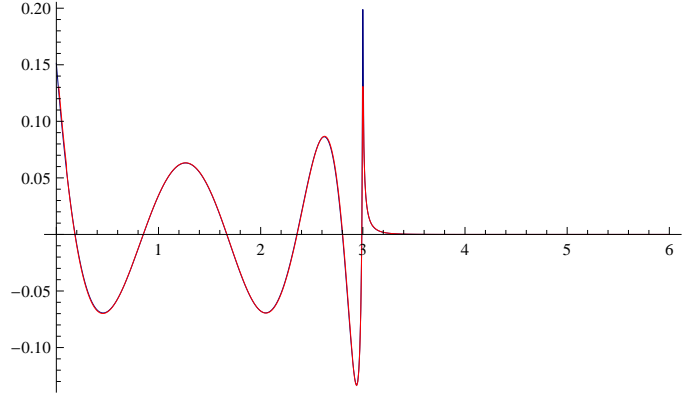}
     \end{center}
     \caption{The singular vector $\mathbf{g_{307}}$ (blue) of $\mathbf{H_T}$ compared with the asymptotics for the singular function $g_6$ (red) of $H_T$. Their good alignment makes them hardly distinguishable.}
         \label{fig:sf-comp}
\end{figure}

\begin{figure}[h!]
     \begin{center}
            \includegraphics[width=0.7\textwidth]{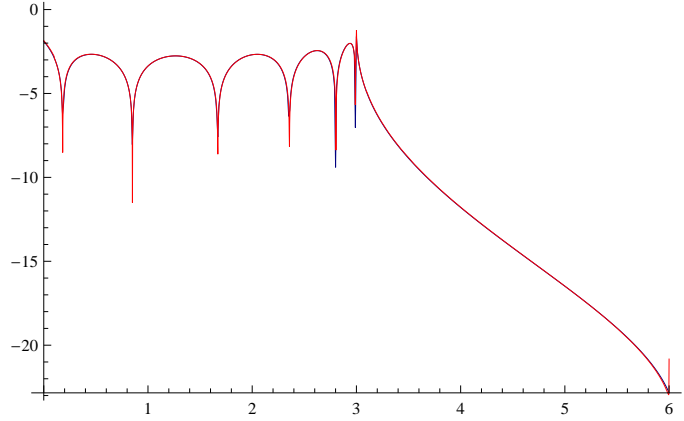}
     \end{center}
     \caption{A logarithmic plot of the comparison in Figure \ref{fig:sf-comp}. This shows very accurate alignment also in the region $[3,6]$ where the functions decay rapidly.}
         \label{fig:sf-comp-log}
\end{figure}

\subsection*{Acknowledgements} We would like to thank Ingrid Daubechies for her valuable comments. RA was supported by a fellowship of the Research Foundation Flanders (FWO) and AK was supported in part by NSF grants DMS-1115615 and DMS-1211164.

\section{Appendix: Normalization of $g$ on $(a_1,a_3)$}\label{norm}
\begin{lemma}
Let $g$ be the solution to $(L-\lambda) \phi =0$ derived in Section \ref{subsec:asymptotics-of-g}. Then,
\begin{equation}\label{normg}
\| g\|_{L^2([a_1,a_3])} = \sqrt{\frac{K_-}{2}} (1+ \mo(\e^{1/2-\delta})).
\end{equation} 
\end{lemma}
\begin{proof}
We want to determine $\int_{a_1}^{a_3} g^2(x) dx$. The main contribution to this integral comes from the WKB solution \eqref{wkba1final} on $[a_1+\mo(\e^{1+2\delta}),a_2-\mo(\e^{1+2\delta})]$. We use the abbreviation $\e_\delta := \e^{1+2\delta}$ and derive
\begin{align*}
\int_{a_1+\mo(\e_\delta)}^{a_2-\mo(\e_\delta)} &g^2(x) dx= \\
&= \int_{a_1+\mo(\e_\delta)}^{a_2-\mo(\e_\delta)} \frac{1}{\sqrt{-P(x)}} \Big[ \cos^2 \Big( \frac{1}{\epsilon} \int_{a_1}^x \frac{dt}{\sqrt{-P(t)}} - \frac{\pi}{4}\Big) +\mo(\e^{1/2-\delta}) \Big] dx \\
&= \int_{a_1+\mo(\e_\delta)}^{a_2-\mo(\e_\delta)} \frac{1}{\sqrt{-P(x)}} \Big[ \frac{1}{2}\cos \Big( \frac{2}{\epsilon} \int_{a_1}^x \frac{dt}{\sqrt{-P(t)}} - \frac{\pi}{4}\Big) +\frac{1}{2}+\mo(\e^{1/2-\delta}) \Big] dx. 
\end{align*}
The first summand in the integral simplifies to 
\begin{align*}
\frac{1}{2} \int_{a_1+\mo(\e_\delta)}^{a_2-\mo(\e_\delta)}  \frac{1}{\sqrt{-P(x)}} & \cos \Big( \frac{2}{\epsilon} \int_{a_1}^x \frac{dt}{\sqrt{-P(t)}} - \frac{\pi}{4}\Big) dx = \\
&= \frac{\e}{4} \sin \Big( \frac{2}{\epsilon} \int_{a_1}^x \frac{dt}{\sqrt{-P(t)}} - \frac{\pi}{4}\Big) \Big\vert_{a_1+\mo(\e_\delta)}^{a_2-\mo(\e_\delta)} \\
&= \mo(\e).
\end{align*}
With that we obtain
\begin{align*}
\int_{a_1+\mo(\e_\delta)}^{a_2-\mo(\e_\delta)} g^2(x) dx &= \mo(\e) + \frac{1}{2} \Big(1+\mo(\e^{1/2-\delta}) \Big) \int_{a_1+\mo(\e_\delta)}^{a_2-\mo(\e_\delta)} \frac{1}{\sqrt{-P(x)}} dx \\
&= \Big(\frac{1}{2}+\mo(\e^{1/2-\delta}) \Big) \int_{a_1+\mo(\e_\delta)}^{a_2-\mo(\e_\delta)} \frac{dx}{\sqrt{-P(x)}}.
\end{align*}
With a Taylor expansion of $1/\sqrt{-P(x)}$, we find that
\begin{align*}
\int_{a_1}^{a_1+\mo(\e_\delta)} \frac{dx}{\sqrt{-P(x)}} &= \frac{1}{\sqrt{-P'(a_1)}} \int_{a_1}^{a_1+\mo(\e_\delta)} \frac{1+\mo(x-a_1)}{\sqrt{x-a_1}} dx =\mo(\e^{1/2+\delta}).
\end{align*}
Similarly,
\begin{align*}
\int_{a_2-\mo(\e_\delta)}^{a_2} \frac{dx}{\sqrt{-P(x)}} &=\mo(\e^{1/2+\delta})
\end{align*}
and thus
\begin{align}
\int_{a_1+\mo(\e_\delta)}^{a_2-\mo(\e_\delta)} g^2(x) dx &= \Big(\frac{1}{2}+\mo(\e^{1/2-\delta}) \Big) \big(K_- + \mo(\e^{1/2+\delta})\big)  \label{norm-g-wkb-a1a2} \\
&= \frac{K_-}{2} \Big(1+\mo(\e^{1/2-\delta}) \Big). \nonumber
\end{align}
Let $t = \lambda(a_1-x)/P'(a_1)$. We consider $g$ in a neighborhood of $a_1^+$, where it can be represented by $g(x) = b_1 \cdot \hat{\psi}_1(x-a_1)$ for $\hat{\psi}_1$ as in \eqref{J0-t-small} for $t \in [0,1)$ and as in \eqref{J0-t-large} for $t \in [1,\mo(\e^{2\delta-1}))]$. Using our previous estimate on the coefficient $b_1$ in \eqref{b1} and a change of variables, we can write
\begin{align}
\int_{a_1}^{a_1+ \mo(\e_\delta)} g^2(x) dx =& b_1^2 \cdot \big\{ \int_0^1 [J_0(2\sqrt{t}) + \mo(\e^{1-2\delta/3})]^2 (-P'(a_1) \e^2) dt \nonumber \\
& + \int_1^{\mo(\e^{2\delta-1})} [J_0(2\sqrt{t}) + t^{-1/4} \cdot \mo(\e^{1-2\delta/3})]^2  (-P'(a_1) \e^2) dt \big\} \nonumber \\
=& \mo(\e) \cdot \big\{ \int_0^{\mo(\e^{2\delta-1})} J_0^2(2\sqrt{t}) dt + \mo(\e^{1-2\delta/3}) + \nonumber \\
&+\mo(\e^{1-2\delta/3} \cdot \e^{3(2\delta-1)/4}) + \mo(\e^{2-4\delta/3} \cdot \e^{\delta-1/2}) \big\} \nonumber \\
=&  \mo(\e) \cdot \big\{ \int_0^{\mo(\e^{2\delta-1})} J_0^2(2\sqrt{t}) dt + \mo(\e^{1/4+5\delta/6})\big\}, \label{norm-g-a1}
\end{align}
where we have used the boundedness of $J_0$ to simplify the error terms. The asymptotic behavior \eqref{asymp-J0-small}, \eqref{asymp-J0-large} of $J_0$ implies that for some constant $c$, $|J_0(u)| \leq \frac{c}{\sqrt{u}}$, for positive $u$. With this we obtain
\begin{equation*}
\int_0^{\mo(\e^{2\delta-1})} J_0^2(2\sqrt{t}) dt \leq  \frac{c^2}{2} \int_0^{\mo(\e^{2\delta-1})} \frac{1}{\sqrt{t}} dt = \mo(\e^{\delta-1/2})
\end{equation*}
and hence
\begin{align}
\int_{a_1}^{a_1+ \mo(\e_\delta)} g^2(x) dx &= \mo(\e^{1/2+\delta}). \label{norm-g-a1-final}
\end{align}
The part of the $L^2$-norm of $g$ in the region at $a_2^-$ can be found in a similar fashion. By matching the WKB and Bessel solutions at $a_2^-$ one can find that $b_2= \mo(\e^{-\delta})$ and $c_2= \mo(\e^{-1/2})$ in
\begin{equation*}
g(x) = b_2 \hat{\psi}_1(x-a_2) + c_2 \hat{\psi}_2(x-a_2)
\end{equation*} 
for $\hat{\psi}_1$ and $\hat{\psi}_2$ as in \eqref{J0-t-small}, \eqref{J0-t-large} and \eqref{Y0-t-small}, \eqref{Y0-t-large}, respectively.
This can also be seen from \eqref{b3-prime}, \eqref{c3-prime}, since the asymptotic behavior of $g$ at $a_2^-$ can be compared to the one of $f$ at $a_3^+$. Replacing $b_1$ by $b_2$ in \eqref{norm-g-a1}, we obtain similarly to \eqref{norm-g-a1-final}
\begin{equation*}
\int_{0}^{\mo(\epsilon_{\delta})} b_2^2 \hat{\psi}_1^2(x) dx = \mo(\e^{3/2-\delta}).
\end{equation*}
This yields
\begin{align*}
&\int_{a_2-\mo(\e_\delta)}^{a_2} g^2(x) dx = \mo(\e^{3/2-\delta}) + \\
& \ \ + b_2 c_2 P'(a_2) \e^2 \big\{ \int_0^1 (J_0(2\sqrt{t}) + \mo(\e^{3/2-\delta/3})) (Y_0(2\sqrt{t}) + \mo(\e^{3/2-\delta/3}))  dt  \\
& \ \ + \int_1^{\mo(\e^{2\delta-1})} (J_0(2\sqrt{t}) + t^{-1/4} \cdot \mo(\e^{1-2\delta/3})) (Y_0(2\sqrt{t}) + t^{-1/4} \cdot \mo(\e^{1-2\delta/3})) dt \big\} \\
& \ \ + c_2^2 P'(a_2) \e^2 \cdot  \big\{ \int_0^1 (Y_0(2\sqrt{t}) + \mo(\e^{3/2-\delta/3}))^2 dt   \\
& \ \ + \int_1^{\mo(\e^{2\delta-1})} (Y_0(2\sqrt{t}) + t^{-1/4} \cdot \mo(\e^{1-2\delta/3}))^2 dt \big\}.
\end{align*}
The asymptotics of $b_2$ and $c_2$ together with the boundedness of $J_0$ allow to simplify the above expression to
\begin{align*}
&\int_{a_2-\mo(\e_\delta)}^{a_2} g^2(x) dx =  \mo(\e^{3/2-\delta}) +\mo(\e^{3/2-\delta})\big\{ \int_0^1 \left|Y_0(2\sqrt{t}) + \mo(\e^{3/2-\delta/3})\right| dt \\
& \ \ + \int_1^{\mo(\e^{2\delta-1})} \left|Y_0(2\sqrt{t}) + t^{-1/4} \cdot \mo(\e^{1-2\delta/3})\right| dt \big\} \\
& \ \ +  \mo(\e) \cdot \big\{ \int_0^{\mo(\e^{2\delta-1})} Y_0(2\sqrt{t})^2 dt +  \mo(\e^{3/2-\delta/3}) \int_0^1 \left| Y_0(2\sqrt{t}) \right| dt \\
& \ \ +  \mo(\e^{3-2\delta/3}) + \mo(\e^{1-2\delta/3}) \int_1^{\mo(2\delta-1)} \left| Y_0(2\sqrt{t}) t^{-1/4} \right| dt + \mo(\e^{1/4+5\delta/6}) \big\} \\
\ \ =& \mo(\e^{3/2-\delta})+\mo(\e^{3/2-\delta})\big\{ \int_0^{\mo(\e^{2\delta-1})} \left|Y_0(2\sqrt{t}) \right| dt + \mo(\e^{1/4+5\delta/6}) \big\} \\
& \ \ +  \mo(\e) \cdot \big\{ \int_0^{\mo(\e^{2\delta-1})} Y_0(2\sqrt{t})^2 dt +  \mo(\e^{3/2-\delta/3}) \int_0^1 \left| Y_0(2\sqrt{t}) \right| dt \\
& \ \ + \mo(\e^{1-2\delta/3}) \int_1^{\mo(2\delta-1)} \left| Y_0(2\sqrt{t}) t^{-1/4} \right| dt + \mo(\e^{1/4+5\delta/6}) \big\}.
\end{align*}
In view of \eqref{asymp-Y0-small} and \eqref{asymp-Y0-large}, there exists a constant $c$ such that $|Y_0(u)| \leq \frac{c}{\sqrt{u}}$ for positive $u$. Thus, we obtain
\begin{align*}
\int_0^{\mo(\e^{2\delta-1})} Y_0(2\sqrt{t})^2 dt &= \mo(\e^{\delta-1/2}), \\
\int_0^{\mo(\e^{2\delta-1})} \left|Y_0(2\sqrt{t})\right| dt &= \mo(\e^{3\delta/2-3/4}), \\
 \int_1^{\mo(2\delta-1)} \left| Y_0(2\sqrt{t}) t^{-1/4} \right| dt &= \mo(\e^{1/2-\delta/3}),
\end{align*}
and hence
\begin{align}
\int_{a_2-\mo(\e_\delta)}^{a_2} g^2(x) dx =& \mo(\e^{1/2+\delta}).\label{norm-g-a2}
\end{align}

The last missing piece is the norm of $g$ on $(a_2,a_3)$, i.e. on the region where it is monotonic. Here, we cannot follow the same procedure as before because the results in Section \ref{bessel-validity} and the corresponding results in \cite{k-t}, were only obtained on the regions where the solution oscillates. 

Instead, we will estimate $\| g \|_{L^2([a_2,a_3])}$ similarly to the derivation in Appendix C, \cite{k-t}. Let $\{ \bar{\lambda}_k; \bar{g}_k \}_{k \in \mathbb{N}}$ be the eigensystem of the following Sturm-Liouville problem:
\begin{equation*}
L \bar{g}_k (x) = \bar{\lambda}_k \bar{g}_k (x), \quad x \in (a_2,a_3),
\end{equation*} 
where the functions $\bar{g}_k(x)$ are bounded at the endpoints $a_2$ and $a_3$. Furthermore, let $g_n$ denote the $n$-th eigenfunction of $\tilde{L}_S$ obtained from the procedure in Section \ref{subsec:asymptotics-of-g} and not normalized yet.

Then, $\chi_{[a_2,a_3]} g_n \in L^2([a_2,a_3])$ can be expanded in the orthonormal basis $\{ \bar{g}_k \}$ of $L^2([a_2,a_3])$:
\begin{equation*}
\chi_{[a_2,a_3]}(x) g_n(x) = \sum_{k \in \mathbb{N}} \langle g_n,\bar{g}_k \rangle \bar{g}_k(x)
\end{equation*}
Let $c_{n,k} = \langle g_n, \bar{g}_k \rangle$.
Then,
\begin{align*}
c_{n,k} =& \frac{1}{\lambda_n} \int_{a_2}^{a_3} (L g_n)(x) \bar{g}_k(x) dx \\
=& \frac{1}{\lambda_n} \int_{a_2}^{a_3} (P g_n')'(x) \bar{g}_k(x) dx + \frac{1}{\lambda_n} \int_{a_2}^{a_3} 2 (x-\sigma)^2 g_n(x) \bar{g}_k(x) dx \\
=& - \frac{1}{\lambda_n} \int_{a_2}^{a_3} (P (x) g_n'(x)) \bar{g}_k'(x) dx + \frac{1}{\lambda_n} \lim_{\e \to 0^+} P g_n' \bar{g}_k \Big|_{a_2+\e}^{a_3} \\
&+ \frac{1}{\lambda_n} \int_{a_2}^{a_3} g_n(x) 2 (x-\sigma)^2 \bar{g}_k(x) dx \\
=& \frac{1}{\lambda_n} \lim_{\e \to 0^+} P (g_n' \bar{g}_k - g_n \bar{g}_k') \Big|_{a_2+\e}^{a_3} + \frac{1}{\lambda_n} \int_{a_2}^{a_3} (P(x) \bar{g}_k'(x))' g_n(x)  dx \\
&+ \frac{1}{\lambda_n} \int_{a_2}^{a_3} g_n(x) 2 (x-\sigma)^2 \bar{g}_k(x) dx. 
\end{align*}
This implies
\begin{align*}
c_{n,k} =& \frac{1}{\lambda_n} \lim_{\e \to 0^+} P(x) [g_n'(x) \bar{g}_k(x) - g_n(x) \bar{g}_k'(x)] \Big|_{a_2+\e}^{a_3} +  \frac{\bar{\lambda}_k}{\lambda_n} c_{n,k}.
\end{align*}
The functions $\bar{g}_k(x)$ are bounded at the endpoints $a_2$ and $a_3$, whereas $g_n(x)$ is bounded at $a_3$ but has a logarithmic singularity at $a_2$. Hence, the above simplifies to
\begin{align*}
c_{n,k} &= - \frac{1}{\lambda_n} \lim_{\e \to 0^+} P(a_2+\e) g_n'(a_2+\e) \bar{g}_k(a_2+\e) +  \frac{\bar{\lambda}_k}{\lambda_n} c_{n,k}, \\
c_{n,k} \big( 1-\frac{\bar{\lambda}_k}{\lambda_n} \big) &= - \frac{1}{\lambda_n} (a_2-a_1) (a_2-a_3) (a_2-a_4) \phi_{2,n} (a_2) \bar{g}_k(a_2), \\
c_{n,k} &= C \phi_{2,n}(a_2) \bar{g}_k(a_2) \frac{1}{\lambda_n - \bar{\lambda}_k}.
\end{align*}
Here $C$ is constant and, close to $a_2$, $g_n$ is of the form 
\begin{equation*}
g_n(x) = \phi_{1,n}(x) + \phi_{2,n}(x) \ln|x-a_2| 
\end{equation*}
for analytic functions $\phi_{i,n}$. Furthermore, $g_n(x)$ satisfies the transmission conditions \eqref{tc1}, \eqref{tc2} at $a_2$ and thus 
\begin{equation*}
\phi_{2,n}(a_2) = \frac{2}{\pi} c_2 = \mo(\sqrt{n}),
\end{equation*}
where we have used $c_2 = \mo(\e^{-1/2})$ as in \eqref{c3-prime}.

One can also find that $\bar{\lambda}_k = \mo(k^2)$ and $\bar{g}_k(a_2) = \mo(\sqrt{k})$, similarly to (5.14) and (6.2) in \cite{k-t}. Note that $\lambda_n \to +\infty$, while $\bar{\lambda}_k \to -\infty$. The norm of $\chi_{[a_2,a_3]} g_n$ can then be found to be 
\begin{align}
\| g_n\|^2_{L^2([a_2,a_3])} &= \sum_{k} c_{n,k}^2 = C^2 \phi_{2,n}^2(a_2) \sum_k \frac{\bar{g}_k^2(a_2)}{(\lambda_n- \bar{\lambda}_k)^2} \nonumber \\
&= \mo(n) \cdot \mo(n^{-2}) = \mo(n^{-1}), \nonumber \\
\| g_n\|_{L^2([a_2,a_3])} &= \mo(n^{-1/2}). \label{norm-g-a2a3}
\end{align}
Putting together \eqref{norm-g-wkb-a1a2}, \eqref{norm-g-a1-final}, \eqref{norm-g-a2} and \eqref{norm-g-a2a3}, we finally obtain
\begin{equation*}
\| g\|_{L^2([a_1,a_3])} = \sqrt{\frac{K_-}{2}} (1+ \mo(\e^{1/2-\delta})).
\end{equation*}
\end{proof}

\bibliographystyle{plain}
\bibliography{Asymptotic-Analysis-THT}

\end{document}